\documentclass[twoside,leqno,symbols-for-thanks]{rmi}
\usepackage[colorlinks,linkcolor=black,citecolor=black,urlcolor=black, linktocpage=true,dvips]{hyperref}
\usepackage{srcltx}
\usepackage{tikz}
\usepackage[english]{babel}
\numberwithin{equation}{section}
\setinitialpage{1}
\renewcommand{\qedsymbol}{$\Box$}
\newcommand{\R}{{\mathbb R}}
\newcommand{\eps}{\varepsilon}
\newcommand{\psl}{(-\Delta)^s_p}
\newcommand{\beq}{\begin{equation}}
\newcommand{\eeq}{\end{equation}}

\newcommand{\T}{{\rm Tail}}
\newenvironment{enumroman}{\begin{enumerate}
\renewcommand{\theenumi}{$(\roman{enumi})$}
}{\end{enumerate}}
\def\Xint#1{\mathchoice
{\XXint\displaystyle\textstyle{#1}}%
{\XXint\textstyle\scriptstyle{#1}}%
{\XXint\scriptstyle\scriptscriptstyle{#1}}%
{\XXint\scriptscriptstyle\scriptscriptstyle{#1}}%
\!\int}
\def\XXint#1#2#3{{\setbox0=\hbox{$#1{#2#3}{\int}$ }
\vcenter{\hbox{$#2#3$ }}\kern-.6\wd0}}

\def\dashint{\Xint-}

\DeclareMathOperator*{\essinf}{ess\,inf}

\numberwithin{equation}{section}
\newtheorem{theorem}{Theorem}[section]
\newtheorem{lemma}[theorem]{Lemma}

\newtheorem{proposition}[theorem]{Proposition}
\newtheorem{corollary}[theorem]{Corollary}
\newtheorem{definition}[theorem]{Definition}
\theoremstyle{definition}
\newtheorem{remark}[theorem]{Remark}

\title[Global H\"older regularity for the fractional $p$-Laplacian]{Global H\"older regularity\\ for the fractional $p$-Laplacian}

\author[A. Iannizzotto, S. Mosconi and M. Squassina]{Antonio Iannizzotto, Sunra Mosconi and Marco Squassina}

\address[antonio.iannizzotto@unica.it]{{\sc A. Iannizzotto}: Dipartimento di Matematica e Informatica, Universit\`a degli Studi di Cagliari, Viale L.\ Merello 92, 09123 Cagliari, Italy}

\address[sunrajohannes.mosconi@univr.it]{{\sc S. Mosconi}: Dipartimento di Informatica, Universit\`a degli Studi di Verona, Strada le Grazie 15, I-37134 Verona, Italy}

\address[marco.squassina@univr.it]{{\sc M. Squassina}: Dipartimento di Informatica, Universit\`a degli Studi di Verona, Strada le Grazie 15, I-37134 Verona, Italy}

\amsclassification[]{35D10, 35R11, 47G20}


\begin{document}

\begin{abstract}
By virtue of barrier arguments we prove $C^\alpha$-regularity up to the boundary for the weak solutions of a non-local, non-linear problem driven by the fractional $p$-Laplacian operator. The equation is boundedly inhomogeneous and the boundary conditions are of Dirichlet type.\ We employ different methods according to the singular ($p<2$) of degenerate ($p>2$) case.
\end{abstract}

\section{Introduction and main result}\label{sec1}

\noindent
We study H\"older regularity up to the boundary for the weak solutions of the Dirichlet problem
\beq\label{dir}
\begin{cases}
(-\Delta)_p^s u=f & \text{in $\Omega$} \\
u=0 & \text{in $\Omega^c$}.
\end{cases}
\eeq
Here $\Omega\subset\R^N$ ($N>1$) is a bounded domain with a $C^{1,1}$ boundary $\partial\Omega$, $\Omega^c=\R^N\setminus\Omega$, $s\in(0,1)$ and $p\in(1,\infty)$ are real numbers and $f\in L^\infty(\Omega)$. The $s$-fractional $p$-Laplacian operator is the gradient of the functional
\[J(u):=\frac{1}{p}\int_{\R^N\times\R^N}\frac{|u(x)-u(y)|^p}{|x-y|^{N+ps}}\, dx\, dy,\]
defined on
\[W^{s,p}_0(\Omega):=\{u\in L^p(\R^N): J(u)<\infty, \,\, u=0 \text{ in $\Omega^c$}\},\]
which is a Banach space with respect to the norm $J(u)^{1/p}$. Under suitable smoothness conditions on $u$ the operator can be written as
\begin{equation*}
(- \Delta)_p^s\, u(x) = 2 \lim_{\varepsilon \searrow 0} \int_{B_\varepsilon^c(x)} \frac{|u(x) - u(y)|^{p-2}\, (u(x) - u(y))}{|x - y|^{N+sp}}\, dy, \quad x \in \R^N.
\end{equation*}
A weak solution $u\in W^{s,p}_0(\Omega)$ of problem \eqref{dir} satisfies,
for every $\varphi\in W^{s,p}_0(\Omega)$, 
\[\int_{\R^N\times\R^N}\frac{|u(x)-u(y)|^{p-2}(u(x)-u(y))(\varphi(x)-\varphi(y))}{|x-y|^{N+ps}} \,dx\,dy=\int_\Omega f(x)\varphi(x) \,dx.\]
Problem \eqref{dir} is thus well posed and, in the case $p=2$, it corresponds to an inhomogeneous fractional Laplacian equation with Dirichlet boundary condition.  
For the sake of completeness we recall that in the literature the fractional Laplacian is often defined by
\begin{equation*}
 \langle (-\Delta)^su,\varphi\rangle
  = \frac{c(N,s)}{2}  \int_{\R^N\times\R^N}\frac{(u(x)-u(y))(\varphi(x)-\varphi(y))}{|x-y|^{N+2s}} \,dx\,dy,\quad \varphi \in  W^{s,2}_0(\Omega),
\end{equation*}
where $c(N,s) = s 2^{2s} \, \Gamma((N+2s)/2) / (\pi^{N/2} \Gamma(1-s))$, in order to be coherent with the
Fourier definition of $(-\Delta)^s$ (see Remark 3.11 of \cite{CS1}). We point out that, in the current literature, there are several
notions of fractional Laplacian, all of which agree when the problems
are set on the whole $\R^N$, but some of them disagree in a bounded
domain.  We refer the reader to~\cite{SV2} for a discussion on the comparison
between the integral fractional laplacian and the regional (or spectral) notion
obtained by taking the $s$-powers of the Laplacian operator $-\Delta$ with zero Dirichlet
boundary conditions.
\vskip2pt
\noindent
In the case $p\neq 2$, problem \eqref{dir} is a non-local and non-linear one. 
Its leading term $\psl$ is furthermore degenerate when $p>2$ and singular when $1<p<2$. Determining sufficiently good regularity estimates {\em up to the boundary} is not only relevant by itself, but it also has useful applications in obtaining multiplicity results for more general non-linear and non-local equations, such as those investigated in \cite{ILPS} in the framework of topological methods and Morse theory. To this regard, 
this contribution provides a first step in order to obtain the results of \cite{IMS} in the general case $p\neq 2$.
\vskip2pt
\noindent
The regularity up to the boundary of fractional problems in the case $p=2$ is now rather 
well understood, even when more general kernels and nonlinearities are considered.  Using a viscosity solution approach, the model linear case gives regularity for fully non-linear equations which are ``uniformly elliptic" in a suitable sense.
Regarding the viscosity approach to fully non-linear, elliptic non-local equation, see \cite{CafSil1,CafSil2} for interior regularity theory with smooth kernels, and \cite{Serra} for rough kernels; regarding boundary regularity, see \cite{RS3} for nearly optimal results and a detailed discussion on the delicate role that the kernel's regularity class plays in such problems.
\vskip2pt
\noindent
Equation \eqref{dir}, however, does not fall in the category of non-local non-linear equations treated in the aforementioned works. This is not surprising, due to the degenerate/singular nature of the nonlinearity, and the $s$-fractional $p$-Laplacian is the non-local analogue of a degenerate/singular non-linear {\em divergence form} equation, rather than of a uniformly elliptic fully non-linear one. Local H\"older continuity has been addressed in \cite{DKP1,DKP2} using methods {\em \'a la} De Giorgi, and in \cite{Ling} with a Krylov-Safanov approach for $p>1/(1-s)$. In \cite{BCF} the fully non-linear approach is used to study the non-local analogue of the $p$-Laplacian equation in non-divergence form
\[ \Delta u +(p-2)\frac{\nabla u}{|\nabla u|}D^2u\frac{\nabla u}{|\nabla u|}=0,\]
arising from non-local `tug of war' games. Interior $C^{1,\alpha}$ estimates and H\"older continuity up to the boundary is proved under rather general assumptions.
\vskip2pt
\noindent
Our main result is the following:

\begin{theorem}\label{main}
There exist $\alpha\in(0,s]$ and $C_\Omega>0$, depending only on $N$, $p$, $s$, with $C_\Omega$ also depending on $\Omega$, such that, for all weak solution $u\in W^{s,p}_0(\Omega)$ of problem \eqref{dir}, $u\in C^\alpha(\overline\Omega)$ and
\beq\label{thm57tesi}
\|u\|_{C^\alpha(\overline\Omega)}\le C_\Omega\|f\|_{L^\infty(\Omega)}^\frac{1}{p-1}.
\eeq
\end{theorem}

\noindent
Notice that, regarding regularity {\em up to the boundary}, one cannot expect more than $s$-H\"older continuity due to explicit examples (see Section \ref{sec3} below). On the other hand, the optimal H\"older exponent up to the boundary seems to be $s$ for any $p>1$, while we prove $C^\alpha$ regularity for an unspecified small $\alpha$, the issue being a lack of higher (at least $C^s$) regularity results in the interior of the domain.
\vskip2pt
\noindent
Let us describe the strategy to prove Theorem \ref{main}. We choose to use the notion of weak rather than viscosity solution, since we feel that the equation is more naturally seen as a variational one. However, we will frequently use barrier arguments, rather than De Giorgi-Nash-Moser techniques. Indeed, the proof of Theorem \ref{main} is performed in the spirit of Krylov's approach to boundary regularity, see \cite{krylov}, and uses two main ingredients:
\begin{itemize}
\item[$(a)$]
a uniform H\"older control (see Theorem \ref{estid}) on how $u$ reaches its boundary values, which amounts to 
\beq\label{deltasintro}
|u(x)|\leq C\|f\|_\infty^{\frac{1}{p-1}}{\rm dist}^s(x,\Omega^c);
\eeq
\item[$(b)$]
a local regularity estimate (see Theorem \ref{osc}) in terms of quantities which may blow up in general when reaching 
the boundary, but remain bounded for functions satisfying \eqref{deltasintro}.
\end{itemize}
Point $(a)$ is obtained through a barrier argument, and stems from the fact that  $\psl (x_+)^s=0$  in the half line $\R_+$. Notice that for $p\neq 2$ we do not have at our disposal the fractional Kelvin transform, and the concrete calculus of the $s$-fractional $p$-Laplacian even on smooth functions is a prohibitive task, in general. 
Thus constructing upper barriers can be quite technical, and is done as following:
\begin{itemize}
\item
Consider $u_N(x)=(x_N)_+^s$: explicit calculus shows that $\psl u_N=0$ in the half-space $\R^N_+$. We locally deform the half-space to $\Omega^c$ by a  diffeomorphism $\Phi$ close to the identity, and obtain a function $u_N\circ \Phi$ with small $s$-fractional $p$-Laplacian in a small ball $\hat B$ centered at a point of $\partial\Omega$.
\item
The resulting function $u_N\circ \Phi$ can be controlled in $\hat B\cap \Omega$ by distance-like functions from the boundary, and we can modify it to globalize the controls, while keeping the smallness of $\psl(u_N\circ\Phi)$ in $\hat B\cap \Omega$.
\item
We exploit the non-local nature of the equation to add a fixed {\em positive} quantity to $\psl(u_N\circ\Phi)$ in $\hat B\cap \Omega$, by truncation away from $\hat B$. Since $\psl(u_N\circ\Phi)$ is arbitrarily small, its truncation has therefore $s$-fractional $p$-Laplacian bounded from below by a positive constant in $\hat B\cap \Omega$, and provides the local upper barrier.
\end{itemize}
Point $(b)$ is a generalization, in the whole range $p>1$, to non-homogeneous equations of Theorem 1.2 from \cite{DKP1}, and it could be deduced in the case $p>2-s/N$ using the results of \cite{KMS} and in the case $p>1/(1-s)$ using \cite{Ling}. However we choose to prove it with a different approach. Much in the spirit of \cite{S}, rather than considering the non-locality of the equation as an additional technical difficulty to the implementation of the De Giorgi-Moser regularity theory, we use it at our advantage to construct a more elementary proof. It should be noted that we do not employ Caccioppoli-like inequalities, or estimates on $\log u$ (which are the elementary counterpart of John-Nirenberg's lemma). Actually we don't even need a Poincar\'e or Sobolev inequality, which are usually looked at as basic tools for (variational) regularity theory. This feature seems typical of the non-local framework and it should be noted that the proof doesn't  seem to immediately ``pass to the limit to local equations'' as the obtained estimates blow up for  $s\to 1$. 
\vskip2pt
\noindent
Regarding possible developments and generalizations, a first remark regards the choice of the kernel in the non-local operator
\[L(u)={\rm PV}\int_{\R^N}|u(x)-u(y)|^{p-2}(u(x)-u(y))K(x,y)\, dy.\]
Regarding interior regularity, a bound from above and below in terms of the model kernel $|x-y|^{-N-ps}$ seems to suffice to obtain H\"older regularity, due to the results of \cite{DKP1,KMS}. For non-local, fully non-linear, uniformly elliptic equation, higher interior regularity (up to $C^{2, \alpha}$) is proved in \cite{CafSil1,CafSil2,Serra} when the kernel satisfies additional structural and regularity assumption, but no such result is known for the $s$-fractional $p$-Laplacian.  
Regarding regularity up to the boundary things are more subtle. In the uniformly elliptic case ($p=2$), the optimal regularity is $C^s(\overline{\Omega})$ due to the results of \cite{RS3}, but only for a subclass of rough symmetric kernels arising from stable L\'evy processes, of the form
\[K(x, y)=H(x-y),\quad H(z)=\frac{a\big(z/|z|\big)}{|z|^{N+2s}}, \quad 0<\lambda\leq a\leq \Lambda.\] 
Counterexamples show that this is the largest kernel's class where to expect such regularity up to the boundary. However, for any $p>1$, one still expects $C^\alpha(\overline{\Omega})$ regularity for arbitrarily rough symmetric kernels, for a small $\alpha<s$.
\vskip2pt
\noindent
Another point of interest is the H\"older regularity up to the boundary of $u/{\rm dist}^s(x,\Omega^c)$, when $\psl u$ is bounded in $\overline{\Omega}$. This is proven in \cite{RS} for the fractional Laplacian, and in \cite{RS3} for the L\'evy stable fully non-linear, uniformly elliptic non-local equations. While undoubtedly being relevant in light of the applications depicted in \cite{ILPS}, we do not treat this problem here.
\vskip2pt
\noindent
The structure of the paper is as follows:
\begin{itemize}
\item In Section \ref{sec2} we mainly discuss the relationship between weak and strong (i.e., in a suitable principal value sense) solutions of \eqref{dir}. In doing so we clarify how barrier arguments (which are more suited to viscosity solutions) can be applied in the framework of weak solutions of non-linear non-local problems.
\item In Section \ref{sec3} we study the $s$-fractional $p$-Laplacian of distance-related functions, and consider their stability with respect to local diffeomorphisms of the domain.
\item In Section \ref{sec4} we construct some upper barriers, derive $L^\infty$-bounds for solutions of \eqref{dir} and prove estimate \eqref{deltasintro}.
\item In Section \ref{sec5} we tackle the local regularity through a weak Harnack inequality. Then we couple it with \eqref{deltasintro} to prove Theorem \ref{main}.
\end{itemize}
\vskip2pt
\noindent
A short description of the result obtained in the present paper can be found in \cite{IMS-RLM}.

\section{Preliminaries}\label{sec2}

\subsection{Notations and function spaces}

\noindent
Given a subset $A\subseteq\R^N$ we will set $A^c=\R^N\setminus A$ and for $A, B\subseteq\R^N$,
\[{\rm dist}(A, B)=\inf_{x\in A,\,y\in B}|x-y|, \quad \delta_A(x)={\rm dist}(x, A^c),\]
\[{\rm dist}_H(A, B)= \max\Big\{\sup_{x\in A}{\rm dist}(x, B), \sup_{y\in B}{\rm dist}(y, A)\Big\}.\]
For all $x\in\R^N$, $r>0$ we denote by $B_r(x)$, $\overline B_r(x)$, and $\partial B_r(x)$, respectively, the open ball, the closed ball and the sphere centered at $x$ with radius $r$. When the center is not specified, we will understand that it's the origin, e.g. $B_1=B_1(0)$. For all measurable $A\subset\R^N$ we denote by $|A|$ the $N$-dimensional Lebesgue measure of $A$. If $u$ is a measurable function and $A$ is a measurable subset of $\R^N$, we will set for brevity
\[\inf_A u=\essinf_A u,\quad \sup_A u=\essinf_A u.\]
For all measurable $u:\R^N\to\R$ we define
\[[u]_{s,p}=\Big(\int_{\R^N\times\R^N}\frac{|u(x)-u(y)|^p}{|x-y|^{N+ps}} \,dx\,dy\Big)^{\frac{1}{p}},\]
\[\|u\|_{W^{s,p}(\Omega)}=\|u\|_{L^p(\Omega)}+\Big(\int_{\Omega\times\Omega}\frac{|u(x)-u(y)|^p}{|x-y|^{N+ps}}\, dx\, dy\Big)^{\frac{1}{p}}\]
and we will consider the following spaces (see \cite{DPV} for details):
\begin{align*}
&W^{s,p}(\Omega)=\big\{u\in L^p(\Omega):\|u\|_{W^{s,p}(\Omega)}<\infty\big\}, \\
&W^{s,p}_0(\Omega) =\big\{u\in W^{s,p}(\R^N):\,u=0 \ \text{in $\Omega^c$}\big\},\\
&W^{-s,p'}(\Omega)=(W^{s,p}_0(\Omega))^*,
\end{align*}
where the last one is the Banach dual, whose pairing with $W^{s,p}_0(\Omega)$ will be denoted by $\langle\cdot,\cdot\rangle_{s,p,\Omega}$.
We will extensively make use of the following space:
\begin{definition}\label{defwtilde}
Let $\Omega\subseteq \R^N$ be bounded. We set
\[\widetilde{W}^{s,p}(\Omega):=\Big\{u\in L^p_{\rm loc}(\R^N):\,\exists\,U\Supset\Omega \ \text{s.t.}\,\|u\|_{W^{s,p}(U)}+\int_{\R^N}\frac{|u(x)|^{p-1}}{(1+|x|)^{N+ps}}\, dx<\infty\Big\}.\]
If $\Omega$ is unbounded, we set
\[\widetilde{W}^{s,p}_{\rm loc}(\Omega):=\big\{u\in L^p_{\rm loc}(\R^N):\,u\in \widetilde{W}^{s,p}(\Omega')\, \text{for any bounded $\Omega'\subseteq\Omega$}\big\}.\]
\end{definition}

\noindent
We notice that the condition 
\[\int_{\R^N}\frac{|u(x)|^{p-1}}{(1+|x|)^{N+ps}}\, dx<\infty\]
holds if  $u\in L^\infty(\R^N)$ or $[u]_{C^s(\R^N)}<\infty$. The spaces $\widetilde{W}^{s,p}(\Omega)$, $\widetilde{W}^{s,p}_{\rm loc}(\Omega)$ can be endowed with a topological vector space structure as inductive limit, but we will not use it.
For all $\alpha\in(0,1]$ and all measurable $u:\overline\Omega\to\R$ we set
\[[u]_{C^\alpha(\overline\Omega)}=\sup_{x,y\in\overline\Omega,\,x\neq y}\frac{|u(x)-u(y)|}{|x-y|^\alpha},\]
\[C^\alpha(\overline\Omega)=\big\{u\in C(\overline\Omega):\,[u]_{C^\alpha(\overline\Omega)}<\infty\big\},\]
the latter being a Banach space under the norm $\|u\|_{C^\alpha(\overline\Omega)}=\|u\|_{L^\infty(\overline\Omega)}+[u]_{C^\alpha(\overline\Omega)}$. A similar definition is given for $C^{1,\alpha}(\overline\Omega)$. When no misunderstanding is possible, we set for all measurable $D\subset\R^N$, $x\in D$, and all measurable $\psi:D\times D\to\R$
\[{\rm PV}\int_{D}\psi(x, y)\,dy=\lim_{\eps\to 0^+}\int_{D\setminus B_\eps(x)}\psi(x, y)\,dy.\]
For all measurable $u:\R^N\to\R$ we recall that the {\em non-local tail} centered at $x\in\R^N$ with radius $R>0$, introduced in \cite{DKP1}, is defined as
\beq\label{deftail}
\T(u;x,R)=\Big(R^{ps}\int_{B_R^c(x)}\frac{|u(y)|^{p-1}}{|x-y|^{N+ps}}\,dy\Big)^\frac{1}{p-1}.
\eeq
We will also set $\T(u; 0, R)=\T(u; R)$. Unless otherwise stated, the numbers $p>1$ and $s\in(0,1)$ will be fixed as the order of summability and the order of differentiability. By a {\em universal} constant we mean a constant $C=C(N,p,s)$. This dependence will always be omitted, even when other dependencies are present, in which case they are the only ones explicitly stated: for example $C_\Omega$ will denote a constant depending on $N, p, s$, and $\Omega$. During chains of inequalities, universal constants will be denoted by the same letter $C$ even if their numerical value may change from line to line. The same treatment will be used for constants which retain their dependencies from line to line. When needed, we will denote a specific universal constant with a number, e.g.\ $C_1$, $C_2$ {\em et cetera}.

\subsection{Some elementary inequalities} 

For all $a\in\R$, $q>0$, we set
\[a^{q}=|a|^{q-1}a.\]
This notation has great advantages in readability and, for future reference, we recall here some more or less known elementary inequalities about the function $a\mapsto a^q$. We will provide a sketch of proof for the less frequent ones.
\vskip2pt
\noindent
We begin with the well known inequalities
\beq\label{in3}
(a+b)^q\le 2^{q-1}(a^q+b^q) \quad  a,b\ge 0,\,q\ge 1;
\eeq
\beq\label{in2}
(a+b)^q\le a^q+b^q \quad  a,b\ge 0,\,q\in(0,1];
\eeq
\beq\label{in4}
|a^q-b^q|\le q(|a|^{q-1}+|b|^{q-1})|a-b| \quad a,b\in\R,\,q\ge 1,
\eeq
the last one being a trivial consequence of Taylor's formula. We will also use
\beq\label{in6}
a^q-(a-b)^q\le C_M\max\{b,b^q\} \quad  |a|\le M,\,b\ge 0,\,q>0,
\eeq
which follows immediately from \eqref{in2} if $q\in (0,1]$. If $q>1$ we can prove it distinguishing the cases $b\leq M$, where we use \eqref{in4}, and the case $b\geq M$, where we use $a^q-(a-b)^q\leq M^q+2M^q\leq 3b^q$.
We now prove
\beq\label{in7}
(a+b)^q-a^q\le\theta a^q+C_\theta b^q \quad  a,b\ge 0,\,q\ge 1,\,C_\theta\to\infty \ \text{as} \ \theta\to 0^+.
\eeq
Letting  $C_q= 1$ if $q\leq 1$ and $C_q=2^{q-1}$ if $q\geq 1$, \eqref{in3} and \eqref{in2} can be written as
\[(a+b)^q\le C_q(a^q+b^q) \quad  a,b\ge 0,\,q>0.\]
Now \eqref{in7} can be proved using Taylor's formula and Young's inequality:
\begin{align*}
(a+b)^q-a^q &\leq C_q(a^{q-1}+b^{q-1})b= (\theta q'a)^{q-1} \frac{C_q b}{(\theta q')^{q-1}}+C_qb^q\\
&\leq\theta a^q+\frac{1}{q}\Big(\frac{C_q}{(\theta q')^{q-1}}\Big)^q b^q +C_qb^q.
\end{align*}
We prove the  following inequality:
\beq\label{in1}
a^{q}-(a-b)^{q}\ge 2^{1-q}b^{q} \quad a\in\R,\,b\ge 0,\,q\ge 1.
\eeq
We can suppose $b>0$ and consider the function
\[f(t)=t^{q}-(t-b)^{q},\quad f'(t)=q(|t|^{q-1}-|t-b|^{q-1}).\]
Therefore $f$ is positive, increasing for $t>b$ and decreasing for $t<-b$ and thus it's coercive. Since $f'(t)=0$ if and only if $t=b/2$, its global minimum is $f(b/2)=2^{1-q}b^{q}$.
\vskip2pt
\noindent
Finally, we will use the following inequality, holding for all $A,B\subset\R^N$ with $A$ bounded and ${\rm dist}(A, B^c)=d>0$:
\beq\label{lkj}
|x-y|\geq C(A, B)(1+|y|), \quad x\in A,\,y\in B^c.
\eeq

\subsection{Weak and strong solutions}

\noindent
We compare in the following different notions of solutions for equations driven by $\psl$.

\begin{definition}
Let $\Omega$ be bounded, $u\in \widetilde{W}^{s,p}(\Omega)$ and $f\in W^{-s,p'}(\Omega)$. We say that $u$ is a {\em weak solution} of $\psl u=f$ in $\Omega$ if for all $\varphi\in W^{s,p}_0(\Omega)$
\[
\int_{\R^N\times\R^N}\frac{(u(x)-u(y))^{p-1}(\varphi(x)-\varphi(y))}{|x-y|^{N+ps}}\,dx\,dy=\langle f, \varphi\rangle_{s,p,\Omega}
\]
If $\Omega$ is unbounded, we say that $u\in \widetilde{W}^{s,p}_{\rm loc}(\Omega)$ solves $\psl u=f$ (with $f\in W^{-s, p'}(\Omega)$) weakly in $\Omega$ if it does so in any bounded open set $\Omega'\subseteq\Omega$.
\end{definition}

\noindent
The inequality $\psl u\leq f$ weakly in $\Omega$ will mean that
\[\int_{\R^N\times\R^N}\frac{(u(x)-u(y))^{p-1}(\varphi(x)-\varphi(y))}{|x-y|^{N+ps}}\,dx\,dy\le\langle f, \varphi\rangle_{s,p,\Omega}\]
for all $\varphi\in W^{s,p}_0(\Omega)$, $\varphi\ge 0$, and similarly for $\psl u\ge f$. Noticing that $\pm K\in W^{-s,p'}(\Omega)$ for any $K>0$ and any bounded $\Omega$, by $|\psl u|\le K$ weakly in $\Omega$ we mean that both $-K\le\psl u\le K$ weakly in $\Omega$.
\vskip2pt
\noindent
In the following proposition we will prove that $\psl u\in W^{-s,p'}(\Omega)$ if $u\in\widetilde{W}^{s,p}(\Omega)$, which implies that the previous definition makes sense.

\begin{lemma}
\label{remws}
Let $\Omega$ be bounded and $u\in \widetilde{W}^{s, p}(\Omega)$. Then the functional 
\[W^{s,p}_0(\Omega)\ni \varphi\mapsto ( u, \varphi):= \int_{\R^N\times\R^N}\frac{(u(x)-u(y))^{p-1}(\varphi(x)-\varphi(y))}{|x-y|^{N+ps}} \,dx\,dy\]
is finite and belongs to $W^{-s,p'}(\Omega)$.
\end{lemma}

\begin{proof}
Let $U\Supset\Omega$ be such that
\beq\label{Uwtilde}
\|u\|_{W^{s,p}(U)}+\int_{\R^N}\frac{|u(x)|^{p-1}}{(1+|x|)^{N+ps}}\, dx<\infty,
\eeq
and write
\beq\label{<>}
\begin{split}
(u, \varphi)&=\int_{U\times U}\frac{(u(x)-u(y))^{p-1}(\varphi(x)-\varphi(y))}{|x-y|^{N+ps}} \,dx\,dy\\
&\quad +\int_{U\times U^c}\frac{(u(x)-u(y))^{p-1}\varphi(x)}{|x-y|^{N+ps}} \,dx\,dy\\
&\quad -\int_{U^c\times U}\frac{(u(x)-u(y))^{p-1}\varphi(y)}{|x-y|^{N+ps}}\,dx\,dy\\
&=\int_{U\times U}\frac{(u(x)-u(y))^{p-1}(\varphi(x)-\varphi(y))}{|x-y|^{N+ps}} \,dx\,dy\\
&\quad +2\int_{\Omega\times U^c}\frac{(u(x)-u(y))^{p-1}\varphi(x)}{|x-y|^{N+ps}} \,dx\,dy,
\end{split}
\eeq
since ${\rm supp}(\varphi)\subset\overline\Omega$.
The integral in $U\times U$ is finite and continuous with respect to strong convergence of $\varphi\in W^{s,p}_0(\Omega)$ since $u\in W^{s,p}(U)$. For the second term, observe that for a.e.\ $x\in \Omega$ it holds
\beq\label{hhh}
\begin{split}
&\int_{U^c}\frac{|u(x)-u(y)|^{p-1}}{|x-y|^{N+ps}}\,dy\\
&\leq C\Big(|u(x)|^{p-1}\int_{U^c}\frac{1}{|x-y|^{N+ps}}\,dy+\int_{U^c}\frac{|u(y)|^{p-1}}{(|x-y|)^{N+ps}} \,dy\Big)\\
&\leq C\Big(|u(x)|^{p-1}+\int_{\R^N}\frac{|u(y)|^{p-1}}{(1+|y|)^{N+ps}} \,dy\Big),
\end{split}
\eeq
where we used \eqref{lkj}
with $A=\Omega$ and $B=U$. The right hand side of \eqref{hhh} belongs to $L^{p'}(\Omega)$ since $\Omega$ is bounded and $u\in L^p(\Omega)$. Thus the second term in \eqref{<>} is continuous with respect to $L^p(\Omega)$-convergence of $\varphi$. Therefore it is also continuous in $W^{s,p}_0(\Omega)$.
\end{proof}
 
\begin{definition}[Point-wise and strong solutions]
Let $u\in \widetilde{W}^{s,p}_{\rm loc}(\Omega)$ and $f:\Omega\to \R$ be measurable. We say that $u$ is an {\em a.e.\ point-wise} solution of $\psl u=f$ in $\Omega$ if for a.a.\ Lebesgue point $x\in \Omega$ of $u$  it holds
\beq
\label{psl-strong.1}
2\, {\rm PV}\int_{\R^N}\frac{(u(x)-u(y))^{p-1}}{|x-y|^{N+ps}}\, dy=f(x).
\eeq
Moreover, for $f\in L^1_{\rm loc}(\Omega)$ we say that $u$ is a {\em strong} solution of $\psl u=f$ if
\beq
\label{psl-strongg}
2\int_{B_\eps^c(x)}\frac{(u(x)-u(y))^{p-1}}{|x-y|^{N+ps}}\, dy\to f\quad \text{strongly in $L^1_{\rm loc}(\Omega)$, as $\eps\to 0^+$.}
\eeq
\end{definition}

\noindent
Similar definitions are given for sub- and supersolutions. 
\vskip2pt
\noindent
Now we prove that a strong solution is also a weak solution. First, we introduce a more general result, which will be used in the following: we denote by ${\bf D}$ the diagonal of $\R^N\times\R^N$.

\begin{lemma}\label{symmset}
Let $u\in\widetilde{W}^{s,p}_{\rm loc}(\Omega)$. For all $\eps>0$ let $A_\eps\subset\R^N\times\R^N$ be a neighborhood of ${\bf D}$ which satisfies 
\begin{enumroman}
\item\label{symmset1}
$(x,y)\in A_\eps$ for all $(y,x)\in A_\eps$;
\item\label{symmset2}
${\rm dist}_H(A_\eps, {\bf D})\to 0$ as $\eps\to 0^+$.
\end{enumroman}
For all $x\in\R^N$ we set $A_\eps(x)=\{y\in\R^N:(x,y)\in A_\eps\}$ and
\[g_\eps(x)=\int_{A_\eps^c(x)}\frac{(u(x)-u(y))^{p-1}}{|x-y|^{N+ps}}\,dy.\]
If $2g_\eps\to f$ in $L^1_{\rm loc}(\Omega)$, then $u$ is a weak solution of $\psl u=f$ in $\Omega$.
\end{lemma}
\begin{proof}
We can suppose that $\Omega$ is bounded and let $U\Supset\Omega$ be such that \eqref{Uwtilde} holds for $u$, fix $\varphi\in C^\infty_c(\Omega)$ and let $K={\rm supp}(\varphi)$. First we prove that $g_\eps\in L^1(K)$. For all $x\in K$ there exists $\rho>0$ such that $B_\rho(x)\subset A_\eps(x)$, and by a covering argument we may choose $\rho$ independent of $x$ (while $\rho$ depends on $\eps$). Moreover, for all $x\in K$ and $y\in A_\eps^c(x)$ we have $|x-y|\ge C(1+|y|)$ (see \eqref{lkj}). So we can compute
\begin{align*}
\int_K|g_\eps(x)|\,dx &\le C\int_K\int_{A_\eps^c(x)}\frac{|u(x)|^{p-1}}{|x-y|^{N+ps}}\,dy\,dx+C\int_K\int_{A_\eps^c(x)}\frac{|u(y)|^{p-1}}{|x-y|^{N+ps}}\,dy\,dx \\
&\le C\int_K|u(x)|^{p-1}\,dx\int_{B_\rho^c}\frac{1}{|z|^{N+ps}}\,dz+C\int_K\int_{A^c_\eps(x)}\frac{|u(y)|^{p-1}}{(1+|y|)^{N+ps}}\,dy\\
&\le C_\eps\int_U|u(x)|^{p-1}\,dx+C|K|\int_{\R^N}\frac{|u(y)|^{p-1}}{(1+|y|)^{N+ps}}\,dy < \infty.
\end{align*}
Lemma \ref{remws} shows that 
\[\frac{(u(x)-u(y))^{p-1}(\varphi(x)-\varphi(y))}{|x-y|^{N+ps}}\in L^1(\R^N\times \R^N)\]
and thus, through \ref{symmset1}, \ref{symmset2}, and Fubini's theorem we have
\begin{align*}
&\int_{\R^N\times\R^N}\frac{(u(x)-u(y))^{p-1}(\varphi(x)-\varphi(y))}{|x-y|^{N+ps}}\,dx\,dy\\
&\underset{(ii)}{=} \lim_{\eps\to 0^+}\int_{A_\eps^c}\frac{(u(x)-u(y))^{p-1}(\varphi(x)-\varphi(y))}{|x-y|^{N+ps}}\,dy\,dx\\
&\underset{(i)}{=}\lim_{\eps\to 0^+}2\int_K\int_{A_\eps^c(x)}\frac{(u(x)-u(y))^{p-1}}{|x-y|^{N+ps}}\varphi(x)\,dx\,dy\\
&=\lim_{\eps\to 0^+}2\int_Kg_\eps(x)\varphi(x)\,dx.
\end{align*}
Since $2g_\eps\to f$ in $L^1(K)$, the density of $C^\infty_c(\Omega)$ in $W^{s,p}_0(\Omega)$ and Lemma \ref{remws} give the assertion.
\end{proof}

\begin{remark}
As the proof shows, it suffices to assume that the convergence in \eqref{psl-strongg} be in $L^1_{\rm loc}(\Omega)$ weakly. We deliberately choose to assume strong $L^1_{\rm loc}$-convergence since in all subsequent applications this is enough.
\end{remark}

\begin{corollary}
\label{simplw}
Let $u\in \widetilde{W}^{s,p}_{\rm loc}(\Omega)$ be a strong solution of $\psl u=f$ in $\Omega$, with $f\in L^1_{\rm loc}(\Omega)$. Then $u$ is a weak solution of $\psl u=f$ in $\Omega$.
\end{corollary}
\begin{proof}
It follows from Lemma \ref{symmset} with $A_\eps=\{(x,y)\in\R^N\times\R^N:\,|x-y|<\eps\}.$
\end{proof}

\subsection{Some basic properties of $\psl$}

The following result describes a fundamental non-local feature of $\psl$.
\begin{lemma}[Non-local behavior of $\psl$]
\label{psadd}
Suppose $u\in \widetilde{W}^{s,p}_{\rm loc}(\Omega)$ solves $\psl u=f$ weakly, strongly or point-wisely in $\Omega$ for some $f\in L^1_{\rm loc}(\Omega)$. Let $v\in L^1_{\rm loc}(\R^N)$ be such that 
\beq
\label{hyppertv}
{\rm dist}({\rm supp}(v),\Omega)>0,\quad \int_{\Omega^c}\frac{|v(x)|^{p-1}}{(1+|x|)^{N+ps}}\, dx<\infty,
\eeq
and define for a.e.\ Lebesgue point $x\in \Omega$ of $u$
\[h(x)=2\int_{{\rm supp}(v)}\frac{(u(x)-u(y)-v(y))^{p-1}-(u(x)-u(y))^{p-1}}{|x-y|^{N+ps}}\,dy.\]
Then $u+v\in \widetilde{W}^{s,p}_{\rm loc}(\Omega)$ and it solves $\psl(u+v)=f+h$ weakly, strongly or pointwisely respectively in $\Omega$.
\end{lemma}

\begin{proof}
As usual, it suffices to consider the case $\Omega$ bounded, and we first prove that $u+v\in \widetilde{W}^{s,p}(\Omega)$. 
Let $K={\rm supp}(v)$ and $U$ be such that \eqref{Uwtilde} holds for $u$, and suppose without loss of generality that $\Omega\Subset U\Subset K^c$. Clearly $u+v=u$ in $U$, and thus it belongs to $W^{s,p}(U)$. Moreover
\[\int_{\R^N}\frac{|u(x)+v(x)|^{p-1}}{(1+|x|)^{N+ps}}\, dx\leq C\Big(\int_{\R^N}\frac{|u(x)|^{p-1}}{(1+|x|)^{N+ps}}\, dx+\int_{K}\frac{|v(x)|^{p-1}}{(1+|x|)^{N+ps}}\, dx\Big),\]
and the last term is finite due to \eqref{hyppertv}.
With a similar estimate, we see that the integral defining $h$ is finite (due also to \eqref{hyppertv} and  \eqref{lkj}).
Consider now the case where $\psl u=f$ weakly.
Choose $\varphi\in C^\infty_c(\Omega)$ and compute
\begin{align*}
&\int_{\R^N\times\R^N}\frac{(u(x)+v(x)-u(y)-v(y))^{p-1}(\varphi(x)-\varphi(y))}{|x-y|^{N+ps}}\,dx\,dy \\
&=\int_{\Omega\times \Omega}\frac{(u(x)-u(y))^{p-1}(\varphi(x)-\varphi(y))}{|x-y|^{N+ps}}\,dx\,dy \\
&\quad + \int_{\Omega\times \Omega^c}\frac{(u(x)-u(y)-v(y))^{p-1}\varphi(x)}{|x-y|^{N+ps}}\,dx\,dy\\
&\quad -\int_{\Omega^c\times \Omega}\frac{(u(x)+v(x)-u(y))^{p-1}\varphi(y)}{|x-y|^{N+ps}}\,dx\,dy \\
&= \int_{\R^N\times \R^N}\frac{(u(x)-u(y))^{p-1}(\varphi(x)-\varphi(y))}{|x-y|^{N+ps}}\,dx\,dy\\
&\quad -\int_{\Omega \times \Omega^c}\frac{(u(x)-u(y))^{p-1}\varphi(x)}{|x-y|^{N+ps}}\,dx\,dy \\
&\quad +\int_{\Omega^c\times \Omega}\frac{(u(x)-u(y))^{p-1}\varphi(y)}{|x-y|^{N+ps}}\,dx\,dy+2\int_{\Omega\times \Omega^c}\frac{(u(x)-u(y)-v(y))^{p-1}\varphi(x)}{|x-y|^{N+ps}}\,dx\,dy \\
&= \int_\Omega f(x)\varphi(x)\,dx+2\int_{\Omega\times \Omega^c}\frac{(u(x)-u(y)-v(y))^{p-1}-(u(x)-u(y)))^{p-1}}{|x-y|^{N+ps}}\varphi(x)\,dx\,dy \\
&= \int_\Omega(f(x)+h(x))\varphi\,dx,
\end{align*}
where in the end we have used Fubini's theorem. The density of $C^\infty_c(\Omega)$ in $W^{s,p}_0(\Omega)$ allows to conclude.
\vskip2pt
\noindent
Suppose now that $\psl u=f$ strongly or pointwisely in $\Omega$. Let for $x\in V\Subset\Omega$ and $\eps<{\rm dist}(V, \Omega^c)$
\[g_\eps(x)=\int_{B^c_\eps(x)}\frac{(u(x)+v(x)-u(y)-v(y))^{p-1}}{|x-y|^{N+ps}}\,dy.\] 
Using \eqref{hyppertv} we get
\begin{align*}
&g_\eps(x)=\int_{\Omega\setminus B_\eps(x)}\frac{(u(x)-u(y))^{p-1}}{|x-y|^{N+ps}}\,dy+\int_{\Omega^c}\frac{(u(x)-u(y)-v(y))^{p-1}}{|x-y|^{N+ps}}\,dy\\
&= \int_{B_\eps^c(x)}\frac{(u(x)-u(y))^{p-1}}{|x-y|^{N+ps}}\,dy+\int_{K}\frac{(u(x)-u(y)-v(y))^{p-1}-(u(x)-u(y))^{p-1}}{|x-y|^{N+ps}}\,dy.
\end{align*}
Taking the limit for $\eps\to 0^+$ gives the claim in the pointwise case. To show that $\psl (u+v)=f+h$ strongly it suffices to show that
\[x\mapsto \int_{K}\frac{(u(x)-u(y)-v(y))^{p-1}-(u(x)-u(y))^{p-1}}{|x-y|^{N+ps}}\,dy\]
belongs to $L^1(K)$, which can be  done proceeding as in \eqref{hhh} and using \eqref{hyppertv} for the term involving $v$. 
\end{proof}

\noindent
We also recall the well known homogeneity, scaling, and rotational invariance properties of $\psl$. For all $\rho>0$, $M\in O_N$ (the orthogonal group), $v$ measurable, $\Omega\subseteq\R^N$, set
\begin{align*}
v_\rho(x)& =v(\rho x),\quad \rho^{-1}\Omega=\{x/\rho:x\in \Omega\},  \\
v_M(x)&=v(Mx),\quad M^{-1}\Omega =\{M^{-1}x:x\in \Omega\}.
\end{align*}
\begin{lemma}\label{hs}
Let $u\in \widetilde{W}^{s,p}_{\rm loc}(\Omega)$ satisfy $\psl u=f$ weakly in $\Omega$ for some $f\in L^{1}_{\rm loc}(\Omega)$. Then we have
\begin{enumroman}
\item\label{hs.1} for all $h>0$, $\psl (hu)=h^{p-1}f$ weakly in $\Omega$;
\item\label{hs.2} for all $\rho>0$, $u_\rho\in \widetilde{W}^{s,p}(\rho^{-1}\Omega)$ and $\psl u_\rho=\rho^{ps}f_\rho$ weakly in $\rho^{-1}\Omega$;
\item\label{hs.3} for all $M\in O_N$, $u_M\in \widetilde{W}^{s,p}(M^{-1}\Omega)$ and  $\psl u_M=f_M$ weakly in $M^{-1}\Omega$.
\end{enumroman}
\end{lemma}

\noindent
Finally, from Lemma 9 of \cite{LL} we have the following comparison principle for $\psl$.

\begin{proposition}[Comparison Principle]\label{comp}
Let $\Omega$ be bounded, $u,v\in \widetilde{W}^{s,p}(\Omega)$ satisfy $u\le v$ in $\Omega^c$ and, for all $\varphi\in W^{s,p}_0(\Omega)$, $\varphi\ge 0$ in $\Omega$,
\begin{align*}
&\int_{\R^N\times\R^N}\frac{(u(x)-u(y))^{p-1}(\varphi(x)-\varphi(y))}{|x-y|^{N+ps}}\,dx\,dy\\
&\le \int_{\R^N\times\R^N}\frac{(v(x)-v(y))^{p-1}(\varphi(x)-\varphi(y))}{|x-y|^{N+ps}}\,dx\,dy.
\end{align*}
Then $u\le v$ in $\Omega$.
\end{proposition}

\begin{proof}
The proof follows by the arguments of \cite{LL}. It is sufficient to know that both sides 
of the inequality are finite and $(u-v)_+\in W^{s,p}_0(\Omega)$, which is used there as a test function.
By Lemma \ref{remws}, both sides are finite. We claim that $w:=(u-v)_+\in W^{s,p}_0(\Omega)$. Let $U\Supset\Omega$ be as in Definition \ref{defwtilde} for both $u$ and $v$. We split the Gagliardo norm in $\R^N$ as
\begin{align*}
&\int_{\R^N\times\R^N}\frac{|w(x)-w(y)|^p}{|x-y|^{N+ps}}\, dx\, dy\\
&= \int_{U\times U}\frac{|w(x)-w(y)|^p}{|x-y|^{N+ps}}\, dx\, dy+2\int_{\Omega\times U^c}\frac{|w(x)|^p}{|x-y|^{N+ps}}\, dx\, dy
\end{align*}
where we used that $w=0$ in $\Omega^c$ by assumption.
The first term is bounded since $u, v\in W^{s,p}(U)$, which is a lattice. The second term is non-singular since ${\rm dist}(\Omega, U^c)>0$ and using  \eqref{lkj}  we get
\begin{align*}
\int_{\Omega\times U^c}\frac{|w(x)|^p}{|x-y|^{N+ps}}\,dx\,dy
&\leq C_{\Omega, U}
\int_\Omega (|u(x)|^{p}+ |v(x)|^{p})\, dx\int_{\R^N}\frac{1}{(1+|y|)^{N+ps}}\, dy \\
& \leq C_{\Omega, U}\int_\Omega (|u(x)|^{p}+ |v(x)|^{p})\, dx,
\end{align*}
which proves the claim. 
\end{proof}

\subsection{$\psl$ on smooth functions}

Next we show that in the class of sufficiently smooth functions, the $s$-fractional $p$-Laplacian exists strongly (and thus weakly) and is locally bounded. First we recall the following definition of $\psl$, equivalent to \eqref{psl-strong.1} (by a simple change of variable):
\beq\label{psl-strong}
\psl u(x) = {\rm PV}\int_{\R^N}\frac{(u(x)-u(x+z))^{p-1}+(u(x)-u(x-z))^{p-1}}{|z|^{N+ps}}\,dz.
\eeq
Our first lemma displays an estimate which allows us to remove the singularity at $0$, when $u$ is smooth enough:

\begin{lemma}\label{zero}
If $u\in C^{1,\gamma}_{\rm loc}(\Omega)$, $\gamma\in [0,1]$, and $K\subset\Omega$ is compact, then there exist $C_{K, u},R_K>0$ 
such that for all $x\in K$, $z\in B_{R_K}$
\begin{equation*}
\big|(u(x)-u(x+z))^{p-1}+(u(x)-u(x-z))^{p-1}\big| 
\leq 
\begin{cases}
 C_{K, u}|z|^{\gamma+p-1}  & \text{if $p\geq 2$}\\
 C_{K, u}|z|^{(\gamma+1)(p-1)} & \text{if $p< 2$}.
\end{cases}
\end{equation*}
\end{lemma}
\begin{proof}
Since $K$ is compact, we can find $R_K>0$ such that
\[\Omega_K:=\{x\in\R^N:\,{\rm dist}(x,K)\le R_K\}\subset\Omega.\]
Consider first the case $p\geq 2$. Since $u\in C^{1,\gamma}(\Omega_K)$, for all $x\in K$, $z\in B_{R_K}$ there exist $\tau_1,\tau_2\in[0,1]$ with
\begin{align*}
&\big|(u(x)-u(x+z))^{p-1}+(u(x)-u(x-z))^{p-1}\big| \\
&= \big|(Du(x+\tau_1 z)\cdot z)^{p-1}-(Du(x-\tau_2 z)\cdot z)^{p-1}\big| \\
&\le (p-1)\sup_{B_{R_K}(x)}|Du|^{p-2}|z|^{p-2}
\big|\big(Du(x+\tau_1 z)-Du(x-\tau_2 z)\big)\cdot z\big| \\
&\le C\|Du\|_{C^{0,\gamma}(\Omega_{K})}^{p-1}|z|^{\gamma+p-1}.
\end{align*}
If $1<p<2$ then $t\mapsto t^{p-1}$ is globally $(p-1)$-H\"older continuous and in this case we directly have, with the same notation as before,
\begin{align*}
&\big|(u(x)-u(x+z))^{p-1}+(u(x)-u(x-z))^{p-1}\big| \\
&\le C\big|Du(x+\tau_1 z)\cdot z-Du(x-\tau_2 z)\cdot z\big|^{p-1}\\
&\le C\|Du\|_{C^{0,\gamma}(\Omega_{K})}^{p-1}|z|^{\gamma(p-1)}|z|^{p-1},
\end{align*}
which concludes the proof.
\end{proof}

\noindent
The following result shows sufficient conditions to write $\psl u$ as a locally bounded function:

\begin{proposition}[$\psl$ on $C^{1,\gamma}$ functions]\label{weak-strong}
Suppose $\Omega$ is bounded, $u\in \widetilde{W}^{s,p}(\Omega)\cap C^{1,\gamma}_{\rm loc}(\Omega)$, with $\gamma\in [0,1]$ such that
\beq\label{assumptionsgamma}
\gamma>
\begin{cases}
1-p(1-s) & \text{if $p\ge 2$} \\
\displaystyle\frac{1-p(1-s)}{p-1} & \text{if $p<2$}.
\end{cases}
\eeq
Then $\psl u=f$ strongly in $\Omega$ for some $f\in L^\infty_{\rm loc}(\Omega)$
\end{proposition}
\begin{proof}
Let $U$ be as in Definition \ref{defwtilde} for $u$, fix a compact set $K\subset\Omega$ and let $R_K,C_K>0$ be as in Lemma \ref{zero}. Define, for $x\in K$, $\eps>0$, 
\begin{align*}
g_\eps(x)&:=\int_{B_\eps^c}\frac{(u(x)-u(x+z))^{p-1}+(u(x)-u(x-z))^{p-1}}{|z|^{N+ps}}\, dz\\
&=2\int_{B_\eps^c}\frac{(u(x)-u(x-z))^{p-1}}{|z|^{N+ps}}\, dz.
\end{align*}
We claim that $g_\eps$ converges as $\eps\to 0^+$ in a dominated way to some $f\in L^\infty(K)$. We split the integral in one for $z\in B_{R_K}$ and one over $B_{R_K}^c$. For the first one, the previous lemma gives
\[\Big|\frac{(u(x)-u(x+z))^{p-1}+(u(x)-u(x-z))^{p-1}}{|z|^{N+ps}}\Big| \le \frac{C_{K, u}}{|z|^{N+ps-\sigma}},\]
where $\sigma=\gamma+p-1$ if $p\geq 2$ and $\sigma=(\gamma+1)(p-1)$ if $1<p<2$. Notice that, in both cases, we have $ps-\sigma<0$. Due to assumptions \eqref{assumptionsgamma}, the integral is thus non-singular, and it holds
\[\lim_{\eps\to 0}\int_{B_{R_K}\setminus B_\eps}\frac{(u(x)-u(x+z))^{p-1}+(u(x)-u(x-z))^{p-1}}{|z|^{N+ps}}\, dz=:f_1(x),\]
\[\Big|\int_{B_{R_K}\setminus B_\eps}\frac{(u(x)-u(x+z))^{p-1}+(u(x)-u(x-z))^{p-1}}{|z|^{N+ps}}\, dz\Big|\leq \int_{B_{R_K}}\frac{C_{K, u}}{|z|^{N+ps-\sigma}}\, dz,\]
which is a bound independent of $x\in K$ and $\eps>0$. For the integral over $z\in B^c_{R_K}$ we have, as in \eqref{hhh},
\begin{align*}
|f_2(x)|&:=\Big|2\int_{B^c_{R_K}}\frac{(u(x)-u(x+z))^{p-1}}{|z|^{N+ps}}\, dz\Big|\\
&\leq  C_{K, U}\Big(\|u\|_{L^\infty(K)}^{p-1}+\int_{\R^N}\frac{|u(y)|^{p-1}}{(1+|y|)^{N+ps}}\, dy\Big).
\end{align*}
Gathering togheter the two estimates, we get
\[|g_\eps(x)|\leq C_{K, u, U}\quad \text{$\forall x\in K$, $\eps>0$},\]
\[\lim_{\eps\to 0^+}g_\eps(x)= f_1(x)+f_2(x)\quad \forall x\in K,\]
and thus by the dominated convergence theorem $g_\eps\to f_1+f_2$ in $L^1(K)$. 
\end{proof}

\begin{remark}
It is useful to outline the dependence of $\|\psl u\|_\infty$ on $s$ in the previous proposition. Suppose, to fix ideas, that $p\geq 2$ and $u\in C^\infty_c(\R^N)$, so that the domain $\Omega$ has no role. Then, following the proof, we can find a constant $c_N$ depending only on $N$ such that
\[\|\psl u\|_{\infty}\leq c_N\frac{\|u\|_{C^2(\R^N)}^{p-1}}{1-s}.\]
This is in accordance with the well known fact that $(1-s)\psl \to -\Delta_p$ as $s\to 1^-$ (see e.g.\ \cite{ponce}).
\end{remark}

\begin{remark}
\label{remarkL}
Consider the class of functions
\[{\mathcal L}(\Omega)=\{u\in \widetilde{W}^{s,p}(\Omega):\psl u=f\,\,\text{strongly for some $f\in L^\infty_{\rm loc}(\Omega)$}\}.\]
The previous theorem asserts that if $p\geq 2$, then $C^2(\Omega)\subseteq {\mathcal L}(\Omega)$. However, if $1<p<2$, it may be difficult  to find smooth functions (e.g., smooth cut-offs) belonging to ${\mathcal L}(\Omega)$, since the second condition in \eqref{assumptionsgamma} coupled with $\gamma\leq 1$ forces $s<2(p-1)/p$. One may think that this is just a technical limit of the proof, or that requiring higher regularity than $C^2$ could solve the issue. Unfortunately, due to the singular nature of the operator for $1<p<2$, this is not the case: there are smooth functions $u$ such that $\psl u$ (in the strong sense) cannot be pointwise bounded. Consider for example $u(x)=x^2\eta(x)$, where $\eta\in C^\infty_c(\R)$ and $\eta=1$ on $[-1,1]$. Calculating $\psl u(0)$ as a principal value  gives
\[|\psl u(0)|<\infty\quad \Leftrightarrow\quad s<2\frac{p-1}{p}.\]
\end{remark}

\section{Distance functions}\label{sec3}

\noindent
In this section we produce some explicit solutions for $\psl$ in special domains. Then we prove that $\psl\delta^s$ is bounded in a neighborhood of $\partial\Omega$ (here we define $\delta=\delta_\Omega$ as in Section \ref{sec2}). We begin by getting a solution of $\psl u=0$ on the half-line $\R_+$.

\begin{lemma}\label{sol1d}
Set $u_1(x)=x_+^s$ for all $x\in\R$. Then $u_1\in \widetilde{W}^{s,p}_{\rm loc}(\R)$ and it solves $\psl u_1=0$ strongly and weakly in $\R_+$.
\end{lemma}
\begin{proof}
Let $K\subseteq (\rho, \rho^{-1})$ for some $\rho\in(0,1)$. For $x\in K$, $\eps>0$ consider the function
\beq
\label{defg}
g^{(1)}_\eps(x)=\int_{B_\eps^c(x)}\frac{(x^s-y_+^s)^{p-1}}{|x-y|^{1+ps}}\, dy.
\eeq
We claim that $g^{(1)}_\eps\to 0$ uniformly on $K$, as $\eps\to 0^+$. Note that for any $\eps<x$ it holds
\[0<x-\eps<x+\eps<\frac{x^2}{x-\eps}.\]
We split the integral accordingly, namely
\begin{align*}
g^{(1)}_\eps(x)&=\int_{-\infty}^0\frac{(x^s-y_+^s)^{p-1}}{|x-y|^{1+ps}}\, dy
+\int_{x+\eps}^{\frac{x^2}{x-\eps}}\frac{(x^s-y_+^s)^{p-1}}{|x-y|^{1+ps}}\, dy \\
&\quad +\Big(\int_0^{x-\eps}\frac{(x^s-y_+^s)^{p-1}}{|x-y|^{1+ps}}\, dy+\int_{\frac{x^2}{x-\eps}}^{\infty}\frac{(x^s-
y_+^s)^{p-1}}{|x-y|^{1+ps}}\, dy\Big) \\
&= I_1(x)+I_2(x, \eps)+I_3(x, \eps).
\end{align*}
Let us estimate the three terms separately. It holds
\[I_1(x)=\int_{-\infty}^0\frac{(x^s-y_+^s)^{p-1}}{|x-y|^{1+ps}}\, dy=\frac{x^{-s}}{ps}.\]
Regarding the integral between $x+\eps$ and $\frac{x^2}{x-\eps}$, since $\xi\mapsto \xi^s$ is globally $s$-H\"older, we have
\[|I_2(x, \eps)|\leq C\int_{x+\eps}^{\frac{x^2}{x-\eps}}\frac{|x-y|^{s(p-1)}}{|x-y|^{1+ps}}\, dy=\frac{Cx^{-s}}{s}\frac{x^s-(x-\eps)^s}{\eps^s}.\]
Finally,
\begin{align*}
I_3(x, \eps)&=\frac{x^{s(p-1)}}{x^{1+ps}}\Big(\int_0^{x-\eps}\frac{(1-(y/x)^s)^{p-1}}{(1-y/x)^{1+ps}}\, dy-\int_{x^2/(x-\eps)}^{\infty}\frac{((y/x)^s-1)^{p-1}}{(y/x-1)^{1+ps}}\, dy\Big)\\
&\!\!\!\!\!\underset{t=\frac{y}{x}=\xi}{=} x^{-s}\Big(\int_0^{1-\eps/x}\frac{(1-t^s)^{p-1}}{(1-t)^{1+ps}}\, dt-\int_{(1-\eps/x)^{-1}}^{\infty}\frac{(\xi^s-1)^{p-1}}{(\xi-1)^{1+ps}}\, d\xi\Big)\\
&\!\!\!\!\underset{\xi=t^{-1}}{=}x^{-s}\Big(\int_0^{1-\eps/x}\frac{(1-t^s)^{p-1}}{(1-t)^{1+ps}}\, dt-\int_0^{1-\eps/x}\frac{(t^{-s}-1)^{p-1}}{(t^{-1}-1)^{1+ps}}\, \frac{dt}{t^2}\Big)\\
&=x^{-s}\int_0^{1-\eps/x}\frac{(1-t^s)^{p-1}}{(1-t)^{1+ps}}(1-t^{s-1})\, dt\\
&=x^{-s}\Big[\frac{1}{ps}\frac{(1-t^s)^p}{(1-t)^{ps}}\Big]_0^{1-\eps/x}=\frac{x^{-s}}{ps}\Big(\Big(\frac{x^s-(x-\eps)^s}{\eps^s}\Big)^p-1\Big).
\end{align*}
Let for $\eps<x$
\beq
\label{defpsi}
\psi(x, \eps)=\frac{x^s-(x-\eps)^s}{\eps^s},
\eeq
and notice that from the subadditivity of $x\mapsto x^s$ we get $\psi(x, \eps)\leq 1$. Gathering together the three previous estimates we get
\beq
\label{estg}
|g^{(1)}_\eps(x)|\leq Cx^{-s}(\psi(x, \eps)+\psi^p(x, \eps))\leq  Cx^{-s}\psi(x, \eps) \quad \forall x>\eps>0,
\eeq
where $C$ is a universal constant.
Since  $\psi(x, \eps)\to 0$ uniformly on $[\rho, \rho^{-1}]\supseteq K$, as $\eps\to 0^+$, the claim follows. Finally we prove that $u_1\in \widetilde{W}^{s,p}(a,b)$ for any $a<0<b$. We have
\begin{align*}
&\int_{[a,b]\times[a,b]}\frac{|u_1(x)-u_1(y)|^p}{|x-y|^{1+ps}}\, dx\, dy\\
&\,\,\,\,=2\int_0^b\int_0^x\frac{|x^s-y^s|^p}{|x-y|^{1+ps}}\, dy\, dx+2\int_0^bx^{sp}\int_a^0\frac{1}{|x-y|^{1+ps}}\, dy\, dx\\
&\underset{t=y/x}{=}2\int_0^b\int_0^1\frac{|1-t^s|^p}{|1-t|^{1+ps}}\, dt\, dx+\frac{2}{ps}\int_0^bx^{sp}\Big(\frac{1}{x^{ps}}-\frac{1}{|x-a|^{ps}}\Big)\, dx,
\end{align*}
which is readily checked to be finite. The assertion follows through Lemma \ref{remws} and Corollary \ref{simplw}.
\end{proof}

\noindent
Now we study the solution $u(x)=u_1(x_N)$ in the half-space
\[\R^N_+=\{x\in\R^N:\,x_N>0\}.\]

\begin{lemma}\label{solnd}
Set for any $A\in GL_N$ and $x\in \R^N_+$,
\[g_{\eps}(x, A)=\int_{B_\eps^c}\frac{(u_1(x_N)-u_1(x_N+z_N))^{p-1}}{|Az|^{N+ps}}\, dz\]
and $u(x)=u_1(x_N)$.
Then $g_\eps\to 0$ uniformly in any compact $K\subseteq \R^N_+\times GL_N$ and $u\in \widetilde{W}^{s,p}_{\rm loc}(\R^N)$ solves $\psl u=0$ strongly and weakly in $\R^N_+$.
\end{lemma}
\begin{proof}
It suffices to prove the statement for $K=H\times H'$, where $H\subseteq \R^N_+$ and $H'\subseteq GL_N$ are compact (recall that $GL_N$ is open in $\R^{N^2}$). 
To estimate $g_{\eps}$ we use elliptic coordinates. A consequence of the singular value decomposition is that $A S^{N-1}$ is an ellipsoid whose semiaxes are the singular values of $A$, and thus its diameter is $2\|A\|_2$, where the latter is the spectral norm of $A$.
The corresponding elliptic coordinates are uniquely defined by
\[y=\rho\omega, \quad \omega\in AS^{N-1},\quad \rho>0,\]
for $y\in \R^N\setminus\{0\}$. It holds $dy=\rho^{N-1}d\omega\, d\rho$ where $d\omega$ is the surface element of $A S^{N-1}$. Setting 
\[e_A:=A^{-1}e_N,\quad E_A:=\{x\in \R^N:\,x\cdot e_A\geq 0\},\]
we compute, through the change of variable $z=A^{-1}y$,
\begin{align*}
&g_{\eps}(x, A)\\
&=\int_{B^c_\eps}\frac{(u(x)-u(x+z))^{p-1}}{|Az|^{N+ps}}\, dz=|{\rm det} A|^{-1}\int_{AB^c_\eps}\frac{(u(x)-u(x+A^{-1}y))^{p-1}}{|y|^{N+ps}}\, dy\\
&=\int_{AS^{N-1}}\frac{1}{|{\rm det} A||\omega|^{N+ps}}\int_\eps^{\infty}\frac{(u_1(x_N) - u_1(x_N+\omega\cdot e_A\rho))^{p-1}}{\rho^{1+ps}}\, d\rho\, d\omega\\
&=\int_{AS^{N-1}\cap E_A}\!\!\frac{|\omega\cdot e_A|^{1+ps}}{|{\rm det} A||\omega|^{N+ps}}\int_{(-\eps, \eps)^c}\frac{(u_1(x_N)-u_1(x_N+\omega\cdot e_A\rho))^{p-1}}{|\omega\cdot e_A\rho|^{1+ps}}\, d (\omega\cdot e_A\rho)\, d\omega\\
&=\int_{AS^{N-1}\cap E_A}\frac{|\omega\cdot e_A|^{1+ps}}{|{\rm det} A||\omega|^{N+ps}}\,g^{(1)}_{\omega\cdot e_A\eps}(x_N) \, d\omega,
\end{align*}
where $g^{(1)}_{\omega\cdot e_A\eps}$ is defined as in \eqref{defg}. Since $|\omega\cdot e_A|\leq \|A\|_2\|A^{-1}\|_2$, the condition
\[\omega\cdot e_A\eps<x_N\]
holds for $\eps\leq \bar \eps$ where $\bar\eps$ depends only on $H$ and $H'$. For any such $\eps$ we can apply \eqref{estg} to obtain 
\[|g_{\eps}(x, A)|\leq Cx_N^{-s}\int_{AS^{N-1}\cap E_A}\frac{|\omega\cdot e_A|^{1+ps}}{|{\rm det} A||\omega|^{N+ps}}\psi(x_N, \omega\cdot e_A\eps)\, d\omega,\]
where $\psi$ is defined in \eqref{defpsi}. Since $\xi\mapsto \xi^s$ is concave for $0<s<1$, we have
\[s(x_N-t)^{s-1}t\ge x_N^s-(x_N-t)^s,\]
and being $1>s>0$ it follows
\[\frac{\partial \psi(x_N, t)}{\partial t}=s\frac{(x_N-t)^{s-1}t-x_N^s+(x_N-t)^s}{t^{1+s}}\geq 0, \quad \text{for $0< t\le x_N$}.\]
Therefore $\psi(x_N, t)$ is non-decreasing in $t$, thus we get
\begin{align*}
|g_{\eps}(x, A)|&\leq Cx_N^{-s}\psi(x_N, \|A\|_2\|A^{-1}\|_2\eps)\int_{AS^{N-1}}\frac{|\omega\cdot e_A|^{1+ps}}{|{\rm det} A||\omega|^{N+ps}}\, d\omega\\
&\leq Cx_N^{-s}\psi(x_N, \|A\|_2\|A^{-1}\|_2\eps)\int_{S^{N-1}}\frac{|\omega\cdot e_N|^{1+ps}}{|A\omega|^{N+ps}}\, d\omega\\
&\leq Cx_N^{-s}\psi(x_N, \|A\|_2\|A^{-1}\|_2\eps)\|A^{-1}\|_2^{N+ps}.
\end{align*}
Now $ \|A\|_2$ and $\|A^{-1}\|_2$ are bounded on $H'$ from below and above, as well as $x_N$ on $H$, and the uniform convergence follows. As in the previous proof, it is readily checked that $u\in \widetilde{W}^{s,p}(V)$ for any bounded $V$, and the second statement follows as before. 
\end{proof}

\begin{remark}\label{rotate}
Due to rotational invariance, Lemma \ref{solnd} easily extends to any half-space
\[H_e=\{x\in\R^N:\,x\cdot e\ge 0\} \quad (e\in S^{N-1}),\]
simply considering the solution $u(x)=(x\cdot e)_+^s$.
\end{remark}

\noindent
The following lemma gives a control on the behaviour of $\psl (x_N)^s_+$ under a smooth change of variables.
\begin{lemma}[Change of variables]
\label{lemmadiffeo}
Let $\Phi$ be a $C^{1,1}$ diffeomorphism of $\R^N$ such that $\Phi=I$ in $B_r^c$, $r>0$. Then the function $v(x)=(\Phi^{-1}(x)\cdot e_N)_+^s$ belongs to $\widetilde{W}^{s,p}_{\rm loc}(\R^N)$ and is a weak solution of $\psl v=f$ in $\Phi(\R^N_+)$, with
\beq
\label{stimaf}
\|f\|_\infty\leq C(\|D\Phi\|_\infty, \|D\Phi^{-1}\|_\infty, r)\|D^2\Phi\|_\infty.
\eeq
\end{lemma}
\begin{proof}
First we recall that, since $D\Phi$ is globally Lipschitz in $\R^N$ with constant $L>0$, 
then $D^2\Phi(x)$ exists in the classical sense for a.a. $x\in\R^N$, and $\|D^2\Phi\|_{L^\infty(\R^N)}\le L$. Let $J_\Phi(\cdot)=|{\rm det}\,D\Phi(\cdot)|$, $u_1(t)=t_+^s$. Due to Lemma \ref{symmset}, applied with $A_\eps=\{|\Phi^{-1}(x)-\Phi^{-1}(y)|<\eps\}$ it suffices to show that
\[g_\eps(x)=\int_{\{|\Phi^{-1}(x)-\Phi^{-1}(y)|\ge\eps\}}\frac{(v(x)-v(y))^{p-1}}{|x-y|^{N+ps}}\, dy\]
converges in $L^1(K)$ for any compact $K\subseteq \Phi(\R^N_+)$. Changing variables $x=\Phi(X)$, this is equivalent to claiming that
\beq
\label{claimgepsilon}
X\mapsto \int_{B_\eps^c(X)}\frac{(u_1(X_N)-u_1(Y_N))^{p-1}}{|\Phi(X)-\Phi(Y)|^{N+ps}}J_\Phi(Y)\, dY
\eeq
converges as $\eps\to 0$ in $L^1_{\rm loc}(\R^N_+)$. To prove this claim, we write
\beq\label{diffeo1}
\begin{split}
g_\eps(x) &= \int_{B_\eps^c(X)}\frac{(u_1(X_N)-u_1(Y_N))^{p-1}}{|D\Phi(X)(X-Y)|^{N+ps}}h(X, Y)\, dY\\
&\quad +\int_{B_\eps^c(X)}J_\Phi(X)\frac{(u_1(X_N)-u_1(Y_N))^{p-1}}{|D\Phi(X)(X-Y)|^{N+ps}}\, dY,
\end{split}
\eeq
where 
\[h(X, Y)=\frac{|D\Phi(X)(X-Y)|^{N+ps}}{|\Phi(X)-\Phi(Y)|^{N+ps}} J_\Phi(Y)-J_\Phi(X), \quad X\neq Y.\]
We will now prove the following estimate, from which convergence of \eqref{claimgepsilon} will follow:
\beq
\label{diffeo2}
|h(X, Y)|\leq C_\Phi\|D^2\Phi\|_\infty\min\{|X-Y|, 1\},
\eeq
where $C_\Phi$ depends on $N$, $p$, $s$ as well as on $\|D\Phi\|_\infty$, $\|D\Phi^{-1}\|_\infty$ and $r$. Write
\begin{align*}
h(X, Y) &=\frac{|D\Phi(X)(X-Y)|^{N+ps}}{|\Phi(X)-\Phi(Y)|^{N+ps}} (J_\Phi(Y)-J_\Phi(X))\\
&\quad +J_\Phi(X)\Big(\frac{|D\Phi(X)(X-Y)|^{N+ps}}{|\Phi(X)-\Phi(Y)|^{N+ps}} -1\Big)\\
&=:J_1+J_2.
\end{align*}
First observe that using Taylor formula yields
\[\frac{|D\Phi(X)(X-Y)|}{|\Phi(X)-\Phi(Y)|}\leq C\|D\Phi\|_\infty\|D\Phi^{-1}\|_\infty,\]
therefore
\[|J_1| \le \tilde{C}\|D^2\Phi\|_{L^\infty(\R^N)}|X-Y|.\]
To estimate $J_2$, we note that the mapping $t\mapsto t^{(N+ps)/2}$ is smooth in a neighborhood of $1$ and that
\[\lim_{Y\to X}\frac{|D\Phi(X)(X- Y)|^2}{|\Phi(X)-\Phi(Y)|^2}=1,\]
hence
\beq
\label{j2}
|J_2|\leq C_\Phi\Big(\frac{|D\Phi(X)(X-Y)|^{2}}{|\Phi(X)-\Phi(Y)|^{2}} -1\Big).
\eeq
Besides, for all $Y\in\R^N$ there exist $\tau_1,\ldots,\tau_N\in[0,1]$ such that
\[\Phi^i(X)-\Phi^i(Y)=D\Phi^i(\tau_i X+(1-\tau_i)Y)\cdot(X-Y), \quad i=1,\ldots,N,\]
where $\Phi^i$ denotes the $i$-th component of $\Phi$. So we have (still allowing $C_\Phi>0$ to depend on $\|D\Phi\|_{L^\infty(\R^N)}$)
\begin{align*}
&\big||\Phi(X)-\Phi(Y)|^2-|D\Phi(X)(X-Y)|^2\big| \\
&= \big|(\Phi(X)-\Phi(Y)+D\Phi(X)(X-Y))\cdot(\Phi(X)-\Phi(Y)-D\Phi(X)(X-Y))\big| \\
&\le C_\Phi|X-Y|\sum_{i=1}^N|\Phi^i(X)-\Phi^i(Y)-D\Phi^i(X)(X-Y)| \\
&\le C_\Phi|X-Y|^2\sum_{i=1}^N|D\Phi^i(\tau_i X+(1-\tau_i)Y)-D\Phi^i(X)| \\
&\le C_\Phi\|D^2\Phi\|_\infty|X-Y|^3.
\end{align*}
Inserting into \eqref{j2} we obtain
\begin{equation*}
|J_2| \le C_\Phi\frac{\big||D\Phi(X)(X-Y)|^2-|\Phi(X)-\Phi(Y)|^2\big|}{|\Phi(X)-\Phi(Y)|^2} 
\le C_\Phi\|D^2\Phi\|_\infty|X-Y|,
\end{equation*}
which yields
\[|h(X, Y)|\leq C_\Phi\|D^2\Phi\|_{L^\infty(\R^N)}|X-Y|,\quad \text{for all $X, Y\in \R^N$},\]
and thus \eqref{diffeo2} for $|X-Y|\leq 2r$. Assume now $|X-Y|>2r$, then at least one of $X$, $Y$ lies in $\overline B_r^c$. Clearly, if $X,Y\in\overline B_r^c$, then $h(X,Y)=0$. If $X\in\overline B_r$, $Y\in\overline B_r^c$, then for any $1\le i\le N$ we define a mapping $\eta_i\in C^{1,1}([0,1])$ by setting
\[\eta_i(t)=\Phi^i(X+t(Y-X)).\]
It is readily checked that $|\eta''_i|\leq C\|D^2\Phi\|_\infty|X-Y|^2$ for a.e.\ $t\in (0,1)$.
Moreover, if $t\geq 2r/|X-Y|$ then $X+t(Y-X)\in B_r^c$, and since $\Phi=I$ outside $B_r$ it holds 
\[\eta_i(t)=(X+t(Y-X))\cdot e_i\quad \text{for $t\geq \frac{2r}{|X-Y|}$}.\]
Therefore $\eta_i''(t)\equiv 0$ for $t\geq 2r/|X-Y|$ and applying the Taylor formula with integral remainder we have
\begin{align*}
&|\Phi^i(Y)-\Phi^i(X)+D\Phi^i(X)(X-Y)|\\
&=|\eta_i(1)-\eta_i(0)-\eta'_i(0)|\leq \int_0^1|\eta''_i(t)|(1-t)\,dt\\
&\leq \int_0^{2r/|X-Y|}|\eta''_i(t)|(1-t)\,dt\le C_\Phi\|D^2\Phi\|_\infty|X-Y|.
\end{align*}
So we have
\begin{align*}
&|h(X,Y)|\\
&\le \Big|\frac{|D\Phi(X)(X-Y)|^{N+ps}}{|\Phi(X)-\Phi(Y)|^{N+ps}}-1\Big|+|1-J_\Phi(X)| \\
&\le C_\Phi\Big|\frac{|D\Phi(X)(X-Y)|^2-|\Phi(X)-\Phi(Y)|^2}{|\Phi(X)-\Phi(Y)|^2}\Big|+C_\Phi\|D^2\Phi\|_\infty \\
&\le C_\Phi\frac{\big|D\Phi(X)(X-Y)+\Phi(X)-\Phi(Y)\big|}{|\Phi(X)-\Phi(Y)|^2}\\
&\quad \cdot\big|D\Phi(X)(X-Y)-\Phi(X)+\Phi(Y)\big|+C_\Phi\|D^2\Phi\|_\infty \\
&\le \frac{C_\Phi}{|X-Y|}\sum_{i=1}^N\big|D\Phi^i(X)(X-Y)-\Phi^i(X)+\Phi^i(Y)\big|+C_\Phi\|D^2\Phi\|_\infty \\
&\le C_\Phi\|D^2\Phi\|_\infty.
\end{align*}
If $X\in\overline B_r^c$, $Y\in\overline B_r$, we argue in a similar way. Thus \eqref{diffeo2} is achieved for all $X,Y\in\R^N$.
\vskip2pt
\noindent
Let us go back to \eqref{diffeo1}. The first integral can be estimated as follows:
\beq
\label{hsf}
\begin{split}
&\int_{B_\eps^c(X)}\Big|\frac{(u_1(X_N)-u_1(Y_N))^{p-1}}{|D\Phi(X)(X-Y)|^{N+ps}}h(X, Y)\Big|\, dY \\
&\leq C_\Phi\|D^2\Phi\|_\infty\int_{B_\eps^c(X)}\frac{\min\{|X-Y|,1\}}{|X-Y|^{N+s}}\, dY\\
&\leq C_\Phi\|D^2\Phi\|_\infty\Big(\int_\eps^1\frac{1}{t^{s}}\, dt+\int_1^{\infty}\frac{1}{t^{1+s}}\, dt\Big)\\
\noalign{\vskip2pt}
&\leq C_\Phi\|D^2\Phi\|_\infty(\eps^{1-s}+1).
\end{split}
\eeq
The second integral in \eqref{diffeo1} vanishes for $\eps\to 0$, and is estimated through Lemma \ref{solnd}: since $D\Phi(\R^N)$ is a compact subset of $GL_N$, the integral vanishes uniformly in any compact $\Phi^{-1}(K)\subseteq \R^N_+$, and therefore uniformly in any compact $K \subseteq \Phi(\R^N_+)$.  Lemma \ref{symmset} thus gives that $\psl v= f$ weakly in any open bounded $U\subseteq \Phi(\R^N_+)$, 
where 
\[f(x):=2\lim_{\eps\to 0}g_\eps(x).\]
Taking the limit for $\eps \to 0$ in estimate \eqref{hsf}  gives \eqref{stimaf}. 
\end{proof}

\noindent
Finally, we consider a general bounded domain $\Omega$ with a $C^{1,1}$ boundary. First we recall some geometrical properties, which can be found e.g. in \cite{AKSZ} (see figure \ref{geometry}):
\begin{figure}
\centering
\begin{tikzpicture}[scale=1.5]
\clip (-3.7,-0.1) rectangle (3.4,4.1);
\shade [left color=lightgray, right color=white, shading=axis, shading angle=180] (-3.7,0.3) to  [out=45, in=180] (0,2) to [out=0, in=200] (3.4,2) -- (3.4,4) -- (-3.7,4);
\draw (-3.4, 3) node[right]{$\Omega$};
\draw[thick] (-3.7,0.3) to  [out=45, in=180] (0,2) to [out=0, in=200] (3.4,2);		
\draw (0,1)  node[right]{$x_2$} circle (1cm);
\filldraw (0,1) circle (0.5pt);
\draw (1.3,0.5) node {$B_{\rho}(x_2)$};

\draw (1.3,2.5) node{$B_{\rho}(x_1)$};
\draw (0,3)  node[right]{$x_1$} circle (1cm);
\filldraw (0,3) circle (0.7pt);

\filldraw (0,2) circle (0.5pt);
\draw (0,2)  node[below right]{$x_0$};

\draw[thick] (0,2) -- (0, 3);

\end{tikzpicture}
\caption{The interior and exterior balls at $x_0\in \partial\Omega$. For all $x\in [x_0, x_1]$ it holds $\delta(x)=|x-x_0|$.}
\label{geometry}
\end{figure}
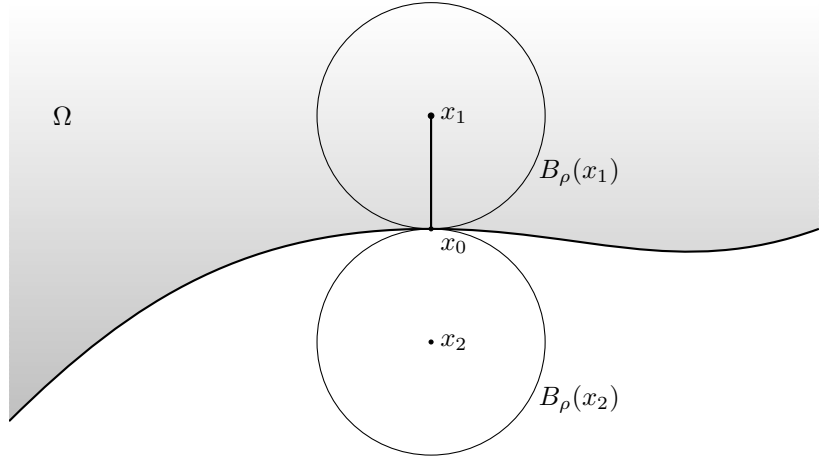

\begin{lemma}\label{geo1}
Let $\Omega\subset\R^N$ be a bounded domain with a $C^{1,1}$ boundary $\partial\Omega$. Then, there exists $\rho>0$ such that for all $x_0\in\partial\Omega$ there exist $x_1,x_2\in\R^N$ on the normal line to $\partial\Omega$ at $x_0$, with the following properties:
\begin{enumroman}
\item $B_{\rho}(x_1)\subset\Omega$, $B_\rho(x_2)\subset\Omega^c$;
\item $\overline B_\rho(x_1)\cap\overline B_\rho(x_2)=\{x_0\}$;
\item $\delta(x)=|x-x_0|$ for all $x\in[x_0,x_1]$.
\end{enumroman}
\end{lemma}

\noindent
As a byproduct, we prove that $\psl\delta^s$ is bounded in a neighborhood of the boundary.

\begin{theorem}\label{deltas}
Let $\Omega\subset\R^N$ be a bounded domain with a $C^{1,1}$ boundary. There exists $\rho=\rho(N, p, s, \Omega)$ such that $(-\Delta)_p^s\delta^s=f$ weakly in 
\[\Omega_\rho:=\{x\in \Omega:\delta(x)<\rho\},\]
for some $f\in L^\infty(\Omega_\rho)$.
\end{theorem}
\begin{proof}
Suppose that $\rho$ is smaller than the one given in Lemma \ref{geo1}. We choose a finite covering of $\Omega_{\rho}$ made of balls of radius $2\rho$ and center $x_i\in \partial\Omega$. Using a partition of unity, it suffices to prove the statement in any set $\Omega \cap B_{2\rho}(x_i)$. To do so, we flatten the boundary near the point $x_i$, which we can suppose without loss of generality to be the origin. Choosing a smaller $\rho$ (depending only on the geometry of $\partial\Omega$) if necessary, there exists a diffeomorphism $\Phi\in C^{1,1}(\R^N,\R^N)$, $\Phi(X)=x$ such that $\Phi=I$ in $B_{4\rho}^c$ and
\beq
\label{hasd}
\Omega\cap B_{2\rho}\Subset \Phi(B_{3\rho}\cap \R^N_+),\quad \delta(\Phi(X))=(X_N)_+,\quad \forall X\in B_{3\rho}.
\eeq
We claim that 
\[g_\eps(x)=\int_{\{|\Phi^{-1}(x)-\Phi^{-1}(y)|\geq \eps\}}\frac{(\delta^{s}(x)-\delta^s(y))^{p-1}}{|x-y|^{N+ps}}\, dy\to f(x)
\quad\text{in $L^1_{\rm loc}(\Omega\cap B_{2\rho})$}.\]
We change variables setting $X=\Phi^{-1}(x)$, noting that $X\in B_{3\rho}\cap \R^N_+$ for any $x\in \Omega\cap B_{2\rho}$, and compute
\begin{align*}
g_\eps(x)&= \int_{\{|X-Y|\geq \eps\}}\frac{(\delta^{s}(\Phi(X))-\delta^s(\Phi(Y)))^{p-1}}{|\Phi(X)-\Phi(Y)|^{N+ps}}J_\Phi(Y)\, dY\\
&=\int_{B_\eps^c(X)\cap B_{3\rho}}\frac{(\delta^{s}(\Phi(X))-\delta^s(\Phi(Y)))^{p-1}}{|\Phi(X)-\Phi(Y)|^{N+ps}}J_\Phi(Y)\, dY\\
&+\int_{B_{3\rho}^c}\frac{(\delta^{s}(\Phi(X))-\delta^s(\Phi(Y)))^{p-1}}{|\Phi(X)-\Phi(Y)|^{N+ps}}J_\Phi(Y)\, dY\\
&=\int_{B_\eps^c(X)}\frac{ (u_1(X_N)-u_1(Y_N))^{p-1}}{|\Phi(X)-\Phi(Y)|^{N+ps}}J_\Phi(Y)\, dY\\
&\quad +\int_{B_{3\rho}^c}\frac{(\delta^{s}(\Phi(X))-\delta^s(\Phi(Y)))^{p-1}-(u_1(X_N)-u_1(Y_N))^{p-1}}{|\Phi(X)-\Phi(Y)|^{N+ps}}J_\Phi(Y)\, dY\\
&= f_{1,\eps}(X)+f_2(X),
\end{align*}
for sufficiently small $\eps$, where we used the fact that
\[\delta^s(\Phi(Z))=u_1(Z_N)\quad \text{for all $Z\in B_{3\rho}$}\]
thanks to \eqref{hasd}.
Clearly $f_2\circ \Phi^{-1}\in L^1(\Omega\cap B_{2\rho})$, and to estimate its $L^\infty$-norm we observe that, due to \eqref{hasd}, 
\[{\rm dist}(\Phi^{-1}(\Omega\cap B_{2\rho}), B_{3\rho}^c)>\theta_{\Phi, \rho}>0.\]
Then, using the $s$-H\"older regularity of $\delta^s\circ\Phi$ and $u_1$, and recalling that $\Phi^{-1}\in {\rm Lip}(\R^N)$ and \eqref{lkj}, we obtain
\begin{align*}
|f_2(X)|&\leq C_{\Phi, \rho}\int_{B^c_{3\rho}}\frac{|X-Y|^{s(p-1)}}{|X-Y|^{N+ps}}\, dY\\
&\leq C_{\Phi, \rho}\int_{\R^N}\frac{1}{(1+|Y|)^{N+s}}\, dY\\
&\leq C_{\Phi, \rho},\quad \forall X\in \Phi^{-1}(\Omega\cap B_{2\rho}).
\end{align*}
Regarding $f_{1,\eps}$, it coincides with the $g_\eps$ of \eqref{diffeo1}. Therefore claim \eqref{claimgepsilon} of Lemma \ref{lemmadiffeo} shows that the limit 
\[f_1(X):=\lim_{\eps\to 0}\int_{B_\eps^c(X)}\frac{(u_1(X_N)-u_1(Y_N))^{p-1}}{|\Phi(X)-\Phi(Y)|^{N+ps}}J_\Phi(Y)\, dY\]
holds in $L^1_{\rm loc}(\R^N_+)$, and $\|f_1\|_{\infty}\leq C_{\Phi, \rho}$. Therefore $g_\eps\to f_1\circ \Phi^{-1}+f_2\circ\Phi^{-1}$ in $L^1_{\rm loc}(\Omega\cap B_{2\rho})$, and both are bounded. Lemma \ref{symmset} finally gives the conclusion.
\end{proof}

\section{Barriers}\label{sec4}

\noindent
In this section we provide some barrier-type functions and prove {\em a priori} bounds for the bounded weak solutions of problem \eqref{dir}. We begin by considering the simple problem
\beq\label{dir1}
\begin{cases}
\psl v=1 & \text{in $B_1$} \\
v=0 & \text{in $B_1^c$.}
\end{cases}
\eeq
The following lemma displays some properties of the solution of \eqref{dir1}:

\begin{lemma}\label{cupola}
Let $v\in W^{s,p}_0(B_1)$ be a weak solution of \eqref{dir1}. Then, $v\in L^\infty(\R^N)$ is unique, radially non-increasing, and for all $r\in(0,1)$ it holds $\inf_{B_r}v>0$.
\end{lemma}
\begin{proof}
First we prove uniqueness. Let the functional $J:W^{s,p}_0(B_1)\to\R$ be defined by
\[J(u)=\frac{1}{p}\int_{\R^N\times\R^N}\frac{|u(x)-u(y)|^p}{|x-y|^{N+ps}}\,dx\,dy-\int_{B_1}u(x)\,dx.\]
$J$ is strictly convex and coercive, hence it admits a unique global minimizer $v\in W^{s,p}_0(B_1)$, which is the only weak solution of \eqref{dir1}. By Lemma \ref{hs} \ref{hs.3} we see that $v$ is radially symmetric, that is, $v(x)=\psi(|x|)$ for all $x\in\R^N$, where $\psi:\R_+\to\R_+$ is a mapping vanishing in $[1,\infty)$. Let $v^\#$ be the symmetric non-increasing rearrangement of $v$. By the fractional P\'olya-Szeg\"o inequality (see Theorem 3 of \cite{B}) we have $J(v^\#)\le J(v)$, so by uniqueness $v=v^\#$, that is, $\psi$ is non-increasing and continuous from the right in $\R_+$. Now let
\[r_0=\inf\{r\in(0,1]:\,\psi(r)=0\}.\]
Clearly $r_0\in(0,1]$. Arguing by contradiction, assume $r_0\in(0,1)$. Then $v\in W^{s,p}_0(B_{r_0})$ and 
it solves weakly
\[\begin{cases}
\psl v=1 & \text{in $B_{r_0}$} \\
v=0 & \text{in $B^c_{r_0}$.}
\end{cases}\]
Reasoning as above and using uniqueness and Lemma \ref{hs} \ref{hs.2}, we see that $v(x)=r_0^{-ps}v(r_0^{ps}x)$ in $B_{r_0}$, so
\[\psi(r_0^2)=r_0^{ps}\psi(r_0)=0,\]
with $r_0^2<r_0$, against the definition of $r_0$. So, for all $r\in(0,1)$ we have
\[\inf_{B_r} v = \psi(r) >0.\]
Finally, we prove that $v\in L^\infty(\R^N)$. Let $w\in C^s(\R^N)\cap \widetilde{W}^{s,p}(B_1)$ be defined by
\[w(x)=\min\{(2-x_N)_+^s,5^s\}.\]
Notice that $w(x)=(2-x_N)_+^s=u_1(2-x_N)$ for all $x\in B_2$. Thus we can apply Lemma \ref{psadd} in $B_{3/2}$ to 
\[w(x)=u_1(2-x_N)-(u_1(2-x_N)-5^s)_+\]
to get, by Lemma \ref{solnd} 
\begin{align*}
\psl w(x) &= 2\int_{\{y_N\le-3\}}\frac{((2-x_N)_+^s-5^s)^{p-1}-((2-x_N)_+^s-(2-y_N)_+^s)^{p-1}}{|x-y|^{N+ps}}\,dy\\
&=: I(x)
\end{align*}
weakly in $B_1$. The function $I:\bar B_1\to\R$ is continuous and positive, so there exists $\alpha>0$ such that
\[\psl w(x)\ge\alpha \ \text{weakly in $B_1$.}\]
We set $\tilde w=\alpha^{-1/(p-1)}w$, so we have
\[\begin{cases}
\psl v=1\le\psl\tilde w & \text{weakly in $B_1$} \\
v=0\le\tilde w & \text{in $B^c_1$,}
\end{cases}\]
and Proposition \ref{comp} yelds
\[0\le v \le\tilde w \le \frac{5^s}{\alpha^\frac{1}{p-1}},\quad \text{in $\R^N$},\]
so $v\in L^\infty(\R^N)$, concluding the proof.
\end{proof}

\noindent
Next we introduce {\em a priori} bounds for functions with bounded fractional $p$-Laplacian.

\begin{corollary}[$L^\infty$-bound]\label{apb}
Let $u\in W^{s,p}_0(\Omega)$ satisfy $|\psl u|\le K$ weakly in $\Omega$ for some $K>0$. Then
\[\|u\|_\infty\le (C_d K)^\frac{1}{p-1},\]
for some $C_d=C(N,p,s, d)$, $d={\rm diam}(\Omega)$.
\end{corollary}
\begin{proof}
Let $v\in W^{s,p}_0(B_1)$ be as in Lemma \ref{cupola}, $x_0$ such that $\Omega\Subset B_d(x_0)$, and set
\[\tilde v(x)=(Kd^{ps})^\frac{1}{p-1}v\Big(\frac{x-x_0}{d}\Big).\]
By Lemma \ref{hs} \ref{hs.1}, \ref{hs.2} we have weakly
\[\begin{cases}
\psl u\le K=\psl\tilde v & \text{in $\Omega$} \\
u=0\le\tilde v & \text{in $\Omega^c$},
\end{cases}\]
which, by Proposition \ref{comp}, implies $u\le\tilde v$ in $\R^N$. A similar argument, applied to $-u$, gives the lower bound.
\end{proof}

\noindent
We can now produce (local) upper barriers on the complements of balls.

\begin{lemma}[Local upper barrier]\label{moon}
There exist $w\in C^s(\R^N)$, and universal $r>0$, $a\in (0,1)$, $c>1$ with
\[\begin{cases}
\psl w\ge a & \text{weakly in $B_r(e_N)\setminus\overline B_1$} \\
c^{-1}(|x|-1)_+^s\le w(x)\le c(|x|-1)_+^s & \text{in $\R^N$}.
\end{cases}\]
\end{lemma}
\begin{proof}
By translation, rotation invariance and scaling (Lemma \ref{hs}), it suffices to prove the statement for any fixed ball of radius $R>2$, at any fixed point $\bar x_R$ of its boundary. To fix ideas, we set $\tilde x_R=(0,-(R^2-4)^{1/2})$ and $\bar x_R=\tilde x_R+Re_N$, so that $B_R(\tilde x_R)$ intersects the hyperplane $\R^{N-1}\times\{0\}$ in the $(N-1)$-ball $\{|x'|<2\}$ (we use the notation $x=(x',x_N)\in\R^{N-1}\times\R$).
\vskip2pt
\noindent
In the following we will choose $R$ large enough, depending only on $N, p, s$. If $R>2$, we can find $\varphi\in C^{1,1}(\R^{N-1})$ such that $\|\varphi\|_{C^{1,1}(\R^{N-1})}\le C/R$ and
\[\varphi(x')=\big((R^2-|x'|^2)^{1/2}-(R^2-4)^{1/2}\big)_+ \ \text{for all $|x'|\in[0,1]\cup[3,\infty)$.}\]
We set
\[U_+=\{x\in\R^N:\,\varphi(x')<x_N\}\]
(see figure \ref{figball}).
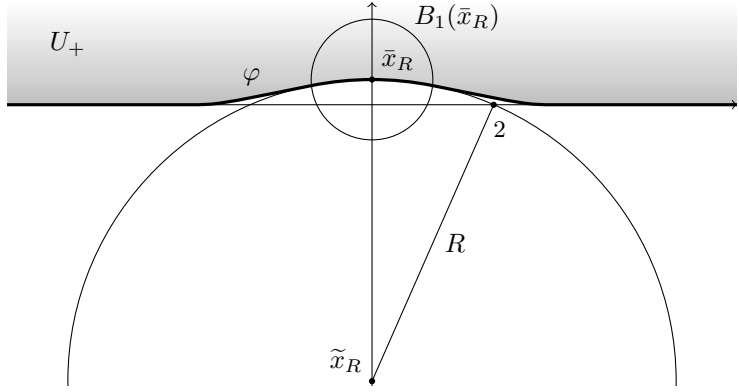
\begin{figure}
\centering
\begin{tikzpicture}[scale=0.8]
\clip (-6, -4.65) rectangle (6, 1.7);

\shade [left color=lightgray, right color=white, shading=axis, shading angle=180] (-6, 0) -- (-3, 0) .. controls (-2.55, 0) .. node[above right=3pt]{$\varphi$} (-1, 0.3164)  arc (101.537:78.463:5)  .. controls (2.55,0) .. (3,0) -- (6,0) -- (6,1.7) -- (-6, 1.7) -- (-6,0);
\draw[very thick] (-6, 0) -- (-3, 0) .. controls (-2.55, 0) ..  (-1, 0.3164)  arc (101.537:78.463:5)  .. controls (2.55,0) .. (3,0) -- (6,0); %phi
\draw[->] (-6, 0) -- (6, 0); %asse x
\draw[->] (0, -5) -- (0, 1.7) ; %asse y
\draw (0, -4.582) circle (5cm); %cerchio_R
\filldraw (0, -4.582) circle (1.2pt); %centro cerchio_R
\draw (0, -4.582) node[above left]{$\widetilde{x}_R$} -- node[right=1pt]{$R$} (2,0);%raggio cerchio_R

\filldraw (2, 0) circle (1.2pt); %punto fisso
\draw (2.1, -0.4) node{\small $2$};

\draw (0, 0.4174) circle (1cm);
\draw (1.4,1.45) node{$B_1(\bar x_R)$};

\draw (-5, 1) node{$U_+$};
\filldraw (0, 0.4174)  circle (1.2pt); %\bar x_R
\draw  (0, 0.4174) node[above right]{$\bar x_R$};

\end{tikzpicture}
\caption{The balls $B_R(\widetilde{x}_R)$ and $B_1(\bar x_R)$. The thick line is the graph of $\varphi$, whose epigraph is $U_+$.}
\label{figball}
\end{figure}
We claim that for any sufficiently large $R$ there exists a diffeomorphism $\Phi\in C^{1,1}(\R^N,\R^N)$ such that $\Phi(0)=\bar x_R$, $\Phi=I$ in $B^c_4$, and
\beq
\label{propPhi}
\|\Phi-I\|_{C^{1,1}(\R^N,\R^N)}+\|\Phi^{-1}-I\|_{C^{1,1}(\R^N,\R^N)}\le\frac{C}{R}, \ \Phi(\R^N_+)=U_+.
\eeq
Indeed, let $\eta\in C^2(\R)$ satisfy $\eta\in [0,1]$, $\eta(0)=1$, ${\rm supp}\,  \eta\subseteq (-1,1)$. Set for all $X=(X',X_N)\in\R^{N-1}\times\R$
\[\Phi(X)=X+\varphi(X')\eta(X_N)e_N.\]
Then, for sufficiently large $R$, $\Phi\in C^{1,1}(\R^N,\R^N)$ is a bijection since $\Phi(X_1)=\Phi(X_2)$ implies $X_1'=X_2'$, and the map $t\mapsto t+\varphi(X')\eta(t)$ is increasing whenever
\[\sup_{X\in \R^N}\varphi(X')|\eta'(X_N)|=\frac{4\sup_{\R}|\eta'|}{R+\sqrt{R^2-4}}< 1.\]
Its inverse mapping satisfies
\beq\label{moon2}
\Phi^{-1}(x)=x-\varphi(x')\eta(\Phi^{-1}(x)\cdot e_N)e_N \ \text{for all $x\in\R^N$,}
\eeq
besides $\Phi(0)=\bar x_R$. Moreover, for all $X\in B^c_4$ we have either $|X'|\ge 3$ or $|X_N|\ge 1$, in both cases $\Phi(X)=X$. The $C^{1,1}$-bounds on $\varphi$, $\eta$ and \eqref{moon2} yield the required $C^{1,1}$-bounds on $\Phi-I$ and $\Phi^{-1}-I$. Finally, reasoning as above, the monotonicity of $t\mapsto t+\varphi(X')\eta(t)$ implies that $\Phi(\R^N_+)=U_+$, and \eqref{propPhi} is proved.
\vskip2pt
\noindent
Let $v_1(x)=u_1(\Phi^{-1}(x)\cdot e_N)$. Lemma \ref{lemmadiffeo} ensures that $v_1\in \widetilde W^{s,p}_{{\rm loc}}(\R^N)$ and
\beq
\label{pslv1}
\psl v_1=f\quad \text{weakly in $U_+$, with $\|f\|_\infty\leq C/R$}.
\eeq
Define 
\[v(x)=\min\{v_1(x), 5^s\},\]
which belongs to $\widetilde{W}^{s,p}(B_4)$. From $\Phi=I$ in $B_4^c$ we infer $\Phi^{-1}(B_4)=B_4$ and thus 
\[v_1(x)=u_1(x_N)\quad\text{in $B_4^c$},\quad v_1\leq 4^s\quad\text{in $B_4$}. \]
Hence
\[v_1(x)-v(x)= (x_N)_+^s-5^s\quad\text{in $\{x_N\geq 5\}$},\quad v_1-v=0\quad\text{in $\{x_N\leq 5\}\Supset B_4$}.\]
Thus the function $v-v_1$ satisfies conditions \eqref{hyppertv} in $B_4$, and Lemma \ref{psadd} provides weakly in $B_4$
\[\psl v=\psl (v_1+(v-v_1))=f+g,\]
where 
\begin{align*}
g(x)&=2\int_{B_4^c}\frac{(v_1(x)-v(y))^{p-1}-(v_1(x)-v_1(y))^{p-1}}{|x-y|^{N+ps}}\, dy\\
&\geq 2\int_{\{y_N\geq 5\}}\frac{((x_N)_+^s-5^s)^{p-1}-((x_N)_+^s-(y_N)_+^s)^{p-1}}{|x-y|^{N+ps}}\, dy
\end{align*}
for any $x\in B_4$. As in the proof of Lemma \ref{cupola}, there is a universal $\alpha>0$ such that $g(x)\geq \alpha$ for all $x\in B_4$, and therefore using \eqref{pslv1} we have
\[\psl v\geq f+g\geq \alpha - \frac{C}{R}\quad \text{weakly in $U_+\cap B_4$}.\]
Taking  $R$ big enough we thus find $B_2(\bar x_R)\Subset B_4$ and 
\beq
\label{moon6}
\psl v\geq \frac{\alpha}{2}>0,\quad \text{weakly in  $U_+\cap B_2(\bar x_R)$}.
\eeq
Set for all $x\in\R^N$
\[d_R(x)=(|x-\tilde x_R|-R)_+.\]
We can estimate $v$ by multiples of $d_R^s$ {\em globally} from above but only {\em locally} from below. Indeed, since $v=0$ in $U_+^c$, $B_R(\tilde x_R)\subset U^c_+$, and $v\in C^s(\R^N)$, there exists $\tilde c>1$ such that
\beq\label{moon7}
v(x)\leq \tilde c\, {\rm dist}(x,U^c_+)^s\leq \tilde c\, d_R^s(x), \,\,\,\,\, \text{for all $x\in \R^N$.}
\eeq
On the other hand, for all $x\in B_{1}(\bar x_R)$ it holds either $x\in B_1(\bar x_R)\setminus U_+\subseteq B_R(\tilde x_R)$, in which case $d_R^s(x)=0=\tilde c v(x)$, or $x\in B_1(\bar x_R)\cap U_+\subseteq B_R^c(\tilde x_R)$. In the latter case let $X=(X', X_N)$ be such that $x=\Phi(X)$, $Z=(X', 0)$ and $z=\Phi(Z)$. It holds $|X'|\leq 1$ and by the construction of $\Phi$, it follows that $z\in \partial B_R(\tilde x_R)$, therefore
\[d_R^s(x)\leq |x-z|^s\le\tilde c|X-Z|^s=\tilde cX_N^s=\tilde cv(x).\]
Thus we have (taking $\tilde c>1$ bigger if necessary)
\beq\label{moon8}
v\ge\frac{1}{\tilde c}d_R^s \ \text{in $B_{1}(\bar x_R)$.}
\eeq
We aim at extending \eqref{moon8} to the whole $\R^N$, while retaining \eqref{moon6} and \eqref{moon7}. For any $\eps\in(0,1/\tilde c)$ set
\[v_\eps=\max\{v,\eps d_R^s\}.\]
Clearly $v_\eps$ satisfies estimates like \eqref{moon7} and \eqref{moon8} in $\R^N$ with a constant $\tilde c_\eps=\max\{\tilde c+\eps,\eps^{-1}\}$. Besides $v\le v_\eps\le v+\eps d_R^s$ in $\R^N$, being $\eps<1/\tilde{c}$, $v_\eps-v=0$ in $B_1(\bar x_R)$. So, by \eqref{moon6}, Lemma \ref{psadd} and \eqref{in6} (with $M=5^s$ and $q=p-1$)
\begin{align*}
\psl v_\eps(x)
&= \psl v(x)-2\int_{B^c_{1/2}(\bar x_R)}\frac{(v(x)-v(y))^{p-1}-(v(x)-v_\eps(y))^{p-1}}{|x-y|^{N+ps}}\,dy \\
&\ge \frac{\alpha}{2}-C\int_{B^c_{1}(\bar x_R)}\frac{\max\{\eps d_R^s(y),(\eps d_R^s(y))^{p-1}\}}{|\bar x_R-y|^{N+ps}}\,dy \\
&\ge \frac{\alpha}{2}-CJ(\eps)
\end{align*}
weakly in $B_{1/2}(\bar x_R)\cap U_+$ (in the end we have used the inequality $|x-y|\ge 1/2|\bar x_R-y|$ for all $x\in B_{1/2}(\bar x_R)$, $y\in B^c_{1}(\bar x_R)$). Notice that $J(\eps)\to 0$ as $\eps\to 0^+$ independently of $x$, thus, for $\eps>0$ small enough we have
\[\psl v_\eps(x)\ge\frac{\alpha}{4}>0 \ \text{weakly in $B_{1/2}(\bar x_R)\setminus B_R(\tilde x_R)$.}\]
To obtain the barrier of the thesis, we set $w(x)=v_\eps(\tilde{x}_R+Rx)$ and using Lemma \ref{hs} we reach the conclusion for $r=1/(2R)$, $a=\alpha/(4R^{ps})$, $c=R^s\max\{\tilde c+\eps,\eps^{-1}\}$.
\end{proof}

\noindent
Finally, we prove that any bounded weak solution of \eqref{dir} can be estimated by means of a multiple of $\delta^s$.

\begin{theorem}\label{estid}
Let $u\in W^{s,p}_0(\Omega)$ satisfy $|\psl u|\le K$ weakly in $\Omega$ for some $K>0$. Then
\beq
\label{thm44tesi}
|u|\le (C_\Omega K)^\frac{1}{p-1}\delta^s \quad \text{a.e.\ in $\Omega$,}
\eeq
for some $C_\Omega=C(N,p,s,\Omega)$.
\end{theorem}
\begin{proof}
Considering $u/K^{1/(p-1)}$ and using homogeneity, we can prove \eqref{thm44tesi} in the case $K=1$. Thanks to Corollary \ref{apb} we may focus on a neighborhood of $\partial\Omega$. Let $\rho>0$ be as in Lemma \ref{geo1}, and let $r\in (0,1)$ be defined in Lemma \ref{moon}. Set
\[U=\Big\{x\in\Omega:\,\delta(x)<r\frac{\rho}{2}\Big\},\]
$\bar x\in U$ and $x_0=\Pi(\bar x)\in \partial\Omega$ its point of minimal distance from $\Omega^c$.
There exists two balls $B_{\rho/2}(x_1)$ and $B_{\rho}(x_2)$ exteriorly tangent to $\partial\Omega$ at $x_0$, and (by scaling and translating the supersolution of the previous Lemma \ref{moon}) a function $w\in C^s$ such that
\beq\label{estid2}
\psl w\ge a  \quad \text{weakly in $B_{r\rho/2}(x_0)\setminus B_{\rho/2}(x_1)$}
\eeq
and
\beq
\label{estid3}
c^{-1}d^s(x)\le w(x)\le cd^s(x)  \quad \text{in $\R^N$}.
\eeq
where we have set
\[d(x)={\rm dist}(x,B^c_{\rho/2}(x_1)).\]
Notice that the constants in \eqref{estid2} \eqref{estid3} depend only on $\rho$, $N$, $p$ and $s$, and we will suppose henceforth that $a, r, c^{-1}\in (0,1)$.
By Lemma \ref{geo1} it holds
\beq\label{estid2b}
d(\bar x)=\delta(\bar x)=|\bar x-x_0|,
\eeq
moreover 
\[d(x)\ge \theta>0,\quad \text{in  $B_{\rho}^c(x_2)\setminus B_{r\rho/2}(x_0)$}.\]
for a constant $\theta$ which depends only on $\rho$ and $r$ (and thus on $\Omega$ alone).
Since $\Omega\subseteq B_{\rho}^c(x_2)$, the latter inequality together with \eqref{estid3} provides
\beq
\label{estid4}
w\geq c^{-1}\theta^s=:\alpha>0,\quad \text{in $\Omega\setminus B_{r\rho/2}(x_0)$}.
\eeq
We define the open set
\[V=\Omega\cap B_{r\frac{\rho}{2}}(x_0)\subseteq B_{r\frac{\rho}{2}}(x_0)\setminus B_{\rho/2}(x_1),\]
where we will apply the comparison principle. Suppose without loss of generality that in \eqref{estid4} $\alpha\in (0,1)$ and let $C_d>1$ be as in Corollary \ref{apb}. Set
\[M=\frac{1}{\alpha}\Big(\frac{C_d}{a}\Big)^\frac{1}{p-1}, \quad \bar w=Mw.\]
By \eqref{estid2} and $C_d/\alpha^{p-1}\geq 1$ we have
\[\psl \bar w=M^{p-1}\psl w\geq \frac{C_d}{\alpha^{p-1}}\geq 1\geq \psl u,\quad \text{weakly in $V$}.\]
Moreover $u=0\leq \bar w$ in $\Omega^c$, while \eqref{estid4}, $a<1$ and Corollary \ref{apb} give 
\[\bar w\geq  M\alpha=\Big(\frac{C_d}{a}\Big)^\frac{1}{p-1}\geq \sup_{\Omega} u,\quad \text{in $\Omega\setminus B_{r\rho/2}(x_0)$}.\] 
Therefore  $\bar w\geq u$ in the whole $V^c$, and Proposition \ref{comp} together with \eqref{estid3} yelds
\[u(x)\leq \bar w(x)\leq cMd^s(x)\quad \text{for a.e.\ $x\in \R^N$}.\]
Recalling \eqref{estid2b} we get
\[u(\bar x)\leq cMd^s(\bar x)=cM\delta^s(\bar x)\quad \text{for all $\bar x=x_0-tn_{x_0}$, $t\in\Big[0, r\frac{\rho}{2}\Big]$},\]
where $n_{x_0}$ is the exterior normal to $\partial\Omega$ at $x_0$, which gives the thesis since $cM$ depends only on $N, p, s$, $\rho$, $r$, and $\Omega$. A similar argument applied to $-u$ yields the lower bound.
\end{proof}

\section{H\"older regularity}\label{sec5}

\noindent
In this section we will obtain the H\"older regularity of solutions.

\subsection{Interior H\"older regularity}

\noindent
We now study the behavior of a weak supersolution in a ball, proving a weak Harnack inequality. Then we will obtain an estimate of the oscillation of a bounded weak solution in a ball (this can be interpreted as a first interior H\"older regularity result). All balls are meant to be centered at $0$, as translation invariance of $\psl$ allows to extend the results to balls centered at any point.
\vskip2pt
\noindent
We begin with a curious Jensen-type inequality:

\begin{lemma}\label{jensen}
Let $E\subset\R^N$ be a set of finite measure and let $u\in L^1(E)$ satisfy
\[\dashint_E u\,dx=1.\]
Then, for all $r\ge 1$ and $\lambda\ge 0$, it holds
\[\dashint_E(u^r-\lambda^r)^\frac{1}{r}\,dx\ge 1-2^\frac{r-1}{r}\lambda.\]
\end{lemma}
\begin{proof}
Avoiding trivial cases, we assume $r>1$ and $\lambda>0$. Set, for all $t\in\R,$
\[g(t)=(t^r-\lambda^r)^\frac{1}{r}.\]
Then, for all $t\in\R\setminus\{0,\lambda\}$, we have
\[g'(t)=|t^r-\lambda^r|^\frac{1-r}{r}|t|^{r-1}.\]
In particular, $t_\lambda=2^{-1/r}\lambda$ is the only solution of $g'(t)=1$. Elementary calculus shows that for all $t\in\R$
\[g(t)\ge g(t_\lambda)+g'(t_\lambda)(t-t_\lambda)=t-2t_\lambda.\]
So we have
\[\dashint_E (g\circ u)\,dx\ge\dashint_E(u-2t_\lambda)\,dx=1-2^\frac{r-1}{r}\lambda,\]
which concludes the proof.
\end{proof}

\noindent
Now we prove a weak Harnack-type inequality for non-negative supersolutions:

\begin{theorem}[Weak Harnack inequality]\label{harnack}
There exist universal $\sigma\in(0,1)$, $\bar C>0$ with the following property: if $u\in \widetilde{W}^{s,p}(B_{R/3})$ satisfies weakly
\[\begin{cases}
\psl u\ge -K & \text{in $B_{R/3}$} \\
u\ge 0 & \text{in $\R^N$,}
\end{cases}\]
for some $K\ge 0$, then
\[\inf_{B_{R/4}}u\ge\sigma\Big(\dashint_{B_R\setminus B_{R/2}}u^{p-1}\,dx\Big)^\frac{1}{p-1}-\bar C(KR^{ps})^\frac{1}{p-1}.\]
\end{theorem}
\begin{proof}
We first consider the case $p\ge 2$. Let $\varphi\in C^\infty(\R^N)$ be such that $0\le\varphi\le 1$ in $\R^N$, $\varphi=1$ in $B_{3/4}$, $\varphi=0$ in $B^c_1$, and by Proposition \ref{weak-strong} $|\psl\varphi|\le C_1$ weakly in $B_1$. We rescale by setting $\varphi_R(x)=\varphi(3x/R)$, so $\varphi_R\in C^\infty(\R^N)$, $0\le\varphi_R\le 1$ in $\R^N$, $\varphi_R=1$ in $B_{R/4}$, $\varphi_R=0$ in $B^c_{R/3}$, and $|\psl\varphi_R|\le C_1R^{-ps}$ weakly in $B_{R/3}$ (taking $C_1$ bigger).
\vskip2pt
\noindent
For all $\sigma\in(0,1)$ we set
\[L=\Big(\dashint_{B_R\setminus B_{R/2}}u^{p-1}\,dx\Big)^\frac{1}{p-1}, \quad w=\sigma L\varphi_R+\chi_{B_R\setminus B_{R/2}} u\]
(see figure \ref{w-barr}).
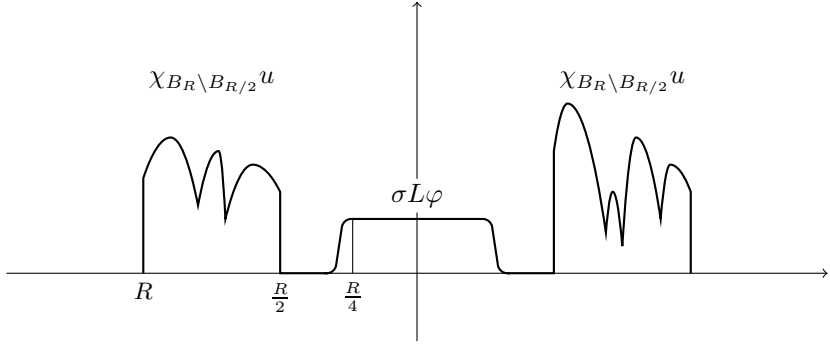
\begin{figure}
\centering
\begin{tikzpicture}[scale=1.8]
\draw[->] (-3, 0) -- (3, 0);%asse x
\draw[->] (0, -0.5) -- (0, 2);%asse y

\draw[thick, rounded corners=3pt]  (-0.68, 0) -- (-0.6, 0) -- (-0.54, 0.4) -- node[above=2pt, fill=white, inner sep=2pt]{$\sigma L\varphi$} (0.54, 0.4) -- (0.6, 0) -- (0.68,0);
\draw (-0.47, 0) node[below]{\small $\frac{R}{4}$} -- (-0.47, 0.4);

\draw[thick] (-2,0) node[below]{\small $R$} -- (-2, 0.7) parabola bend (-1.8, 1) (-1.6,0.5) parabola bend (-1.45, 0.9) (-1.4, 0.4) parabola bend (-1.2, 0.8) (-1,0.6) -- (-1,0) node[below]{\small $\frac{R}{2}$} -- (-0.65, 0) ;
\draw (-1.5,1.4) node{$\chi_{B_R\setminus B_{R/2}}u$};
\draw[thick] (0.66,0) -- (1,0) -- (1,0.9) parabola bend (1.1, 1.25)  (1.38,0.3) parabola bend (1.43, 0.6) (1.5,0.2) parabola bend (1.6,1) (1.78,0.4) parabola bend (1.85, 0.8) (2,0.6) -- (2,0);
\draw (1.5,1.4) node{$\chi_{B_R\setminus B_{R/2}}u$};

\end{tikzpicture}
\caption{The lower barrier $w$.}
\label{w-barr}
\end{figure}
So $w\in \widetilde{W}^{s,p}(B_{R/3})$  and by Lemma \ref{psadd} and \eqref{in1} we have, weakly in $B_{R/3}$,
\begin{align*}
&\psl w(x)\\
&= \psl(\sigma L\varphi_R)(x)+2\int_{B_R\setminus B_{R/2}}\frac{(\sigma L\varphi_R(x)-u(y))^{p-1}-(\sigma L\varphi_R(x))^{p-1}}{|x-y|^{N+ps}} \, dy\\
&\le \frac{C_1(\sigma L)^{p-1}}{R^{ps}}-2^{3-p}\int_{B_R\setminus B_{R/2}}\frac{u^{p-1}(y)}{|x-y|^{N+ps}}\,dy \\
&\le \frac{C_1(\sigma L)^{p-1}}{R^{ps}}-\frac{C_2L^{p-1}}{R^{ps}}.
\end{align*}
If we assume
\[\sigma<\min\Big\{1,\Big(\frac{C_2}{2C_1}\Big)^\frac{1}{p-1}\Big\},\]
we get the upper estimate
\beq\label{hk1}
\psl w(x)\le-\frac{C_2L^{p-1}}{2R^{ps}} \quad  \text{weakly in $B_{R/3}$.}
\eeq
We set $\bar C=(2/C_2)^{1/(p-1)}$ and distinguish two cases:
\begin{itemize}
\item if $L\le\bar C(KR^{ps})^{1/(p-1)}$, then clearly
\[\inf_{B_{R/4}}u\ge 0\ge\sigma L-\bar C(KR^{ps})^\frac{1}{p-1};\]
\item if $L>\bar C(KR^{ps})^{1/(p-1)}$, then we use \eqref{hk1} to have 
\[\begin{cases}
\psl w\le -K\le \psl u & \text{weakly in $B_{R/3}$} \\
w=\chi_{B_R\setminus B_{R/2}} u\le u & \text{in $B^c_{R/3}$,}
\end{cases}\]
which by Proposition \ref{comp} implies $w\le u$ in $\R^N$, in particular
\[\inf_{B_{R/4}}u\ge \sigma L.\]
\end{itemize}
In any case we have
\[\inf_{B_{R/4}}u \ge \sigma L-\bar C(KR^{ps})^\frac{1}{p-1},\]
which is the conclusion.
\vskip2pt
\noindent
Now we consider the case $p\in(1,2)$. Due to Remark \ref{remarkL}, in this case we cannot use the cut-off function $\varphi$ as before to construct the barrier $w$. We use instead the weak solution $v$ of \eqref{dir1} introduced in Lemma \ref{cupola}, recalling that $\inf_{B_{3/4}}v>0$, and we set
\[\varphi_R(x)=\Big(\inf_{B_{3/4}}v\Big)^{-1}v\Big(\frac{3x}{R}\Big),\] 
so that $0\le\varphi_R\le\alpha$ (for some universal $\alpha>0$) in $\R^N$, $\varphi_R\ge 1$ in $B_{R/4}$, $\varphi_R=0$ in $B^c_{R/3}$, and $\psl\varphi_R= C_1R^{-ps}$ weakly in $B_{R/3}$. Accordingly, to obtain the estimate \eqref{hk1} we apply Lemma \ref{jensen} to the function $(u/L)^{p-1}$ with $E=B_R\setminus B_{R/2}$, $r=1/(p-1)$, and $\lambda=(\sigma\varphi_R(x))^{p-1}$, so that
\[\dashint_{B_R\setminus B_{R/2}}\Big(\frac{u(y)}{L}-\sigma\varphi_{R}(x)\Big)^{p-1}\,dy\ge 1-2^{2-p}(\sigma\varphi_R(x))^{p-1},\]
for a.e.\ $x\in B_{R/3}$. This, in turn, implies that for a.e.\ $x\in B_{R/3}$
\begin{align*}
&2\int_{B_R\setminus B_{R/2}}\frac{(\sigma L\varphi_R(x)-u(y))^{p-1}-(\sigma L\varphi_R(x))^{p-1}}{|x-y|^{N+ps}}\, dy\\ 
&\le \frac{C_2}{R^{ps}}\dashint_{B_R\setminus B_{R/2}}(\sigma L\varphi_R(x)-u(y))^{p-1}\,dy \\
&\le \frac{C_2}{R^{ps}}\big(2^{2-p}(\sigma L\varphi_R(x))^{p-1}-L^{p-1}\big) \\
&\le -\frac{C_2L^{p-1}}{2R^{ps}},
\end{align*}
where we have chosen $\sigma<2^{\frac{p-3}{p-1}}\alpha^{-1}$. Then, by taking $\sigma$ even smaller if necessary, we get \eqref{hk1} and the rest of the proof follows {\em verbatim}.
\end{proof}

\noindent
We need to extend Theorem \ref{harnack} to supersolutions which are only non-negative in a ball. To do so, we introduce a tail term defined as in \eqref{deftail}:

\begin{lemma}\label{harnackloc}
There exist $\sigma\in(0,1)$, $\tilde C>0$, and for all $\eps>0$ a constant $C_\eps>0$ with the following property: if $u\in \widetilde{W}^{s,p}(B_{R/3})$ satisfies weakly
\[\begin{cases}
\psl u\ge -K & \text{in $B_{R/3}$} \\
u\ge 0 & \text{in $B_R$,}
\end{cases}\]
for some $K\geq 0$, then
\beq
\label{hloctesi}
\inf_{B_{R/4}}u\ge\sigma\Big(\dashint_{B_R\setminus B_{R/2}}u^{p-1}\,dx\Big)^\frac{1}{p-1}-\tilde C(KR^{ps})^\frac{1}{p-1}-C_\eps{\rm Tail}(u_-;R)-\eps\sup_{B_R}u.
\eeq
\end{lemma}
\begin{proof}
First we consider the case $p\ge 2$. We apply Lemma \ref{psadd} to the functions $u$ and $v=u_-$, so that $u+v=u_+$, and $\Omega=B_{R/3}$: we have in a weak sense in $B_{R/3}$
\begin{align*}
&\psl u_+(x) \\
&= \psl u(x)+2\int_{B^c_{R/3}}\frac{(u(x)-u(y)-u_-(y))^{p-1}-(u(x)-u(y))^{p-1}}{|x-y|^{N+ps}}\,dy \\
&\ge -K+2\int_{\{u<0\}}\frac{u(x)^{p-1}-(u(x)-u(y))^{p-1}}{|x-y|^{N+ps}}\,dy \\
&\ge -K+C\int_{\{u<0\}}\frac{u(x)^{p-1}-(u(x)-u(y))^{p-1}}{|y|^{N+ps}}\,dy,
\end{align*}
where in the end we have used that $|x-y|\ge 2/3|y|$. By \eqref{in7}, for any $\theta>0$ there exists $C_\theta>0$ such that weakly in $B_{R/3}$
\begin{align*}
\psl u_+(x) &\ge -K-\theta\big(\sup_{B_R}u\big)^{p-1}\int_{B^c_R}\frac{1}{|y|^{N+ps}}\,dy-\frac{C_\theta}{R^{ps}}{\rm Tail}(u_-;R)^{p-1} \\
&\ge -K-\frac{C\theta}{R^{ps}}\big(\sup_{B_R}u\big)^{p-1}-\frac{C_\theta}{R^{ps}}{\rm Tail}(u_-;R)^{p-1} =: -\tilde K.
\end{align*}
Now, by applying Theorem \ref{harnack} to $u_+$ we have for any $\eps>0$ and $\theta<(\eps/\bar C)^{p-1}$,
\begin{align*}
\inf_{B_{R/4}}u &\ge \sigma\Big(\dashint_{B_R\setminus B_{R/2}}u^{p-1}\,dx\Big)^\frac{1}{p-1}-\bar C(\tilde KR^{ps})^\frac{1}{p-1} \\
&\ge \sigma\Big(\dashint_{B_R\setminus B_{R/2}}u^{p-1}\,dx\Big)^\frac{1}{p-1}-\tilde C(KR^{ps})^\frac{1}{p-1}-C_\eps{\rm Tail}(u_-;R)-\eps\sup_{B_R}u
\end{align*}
for some universal constant $\tilde C>0$ and a convenient $C_\eps>0$ depending also on $\eps$.
\vskip2pt
\noindent
Now we turn to the case $p\in(1,2)$. The argument in this case is in fact easier, as by \eqref{in2} we have
\[\int_{\{u<0\}}\frac{u(x)^{p-1}-(u(x)-u(y))^{p-1}}{|y|^{N+ps}}\,dy\le\frac{1}{R^{ps}}{\rm Tail}(u_-;R)^{p-1}\quad \text{for a.e.\ $x\in B_{R/3}$},\]
then we argue as above using \eqref{in3} instead of \eqref{in1} when required.
\end{proof}

\noindent
Clearly, symmetric versions of Theorem \ref{harnack} and Lemma \ref{harnackloc} also hold. Now we use the above results to produce an estimate of the oscillation of a bounded function $u$ such that $\psl u$ is locally bounded. We set for all $R>0$, $x_0\in \R^N$
\[Q(u; x_0, R) = \|u\|_{L^\infty(B_R(x_0))}+\T(u; x_0, R),\quad Q(u; R)=Q(u; 0, R).\]
Our result is as follows:

\begin{theorem}\label{osc}
There exist universal $\alpha\in(0,1)$, $C>0$ with the following property: if $u\in \widetilde{W}^{s,p}(B_{R_0})\cap L^\infty(B_{R_0})$ satisfies $|\psl u|\le K$ weakly in $B_{R_0}$ for some $R_0>0$, then for all $r\in(0,R_0)$
\[\underset{B_r}{\rm osc}\,u\le C\big[(KR_0^{ps})^\frac{1}{p-1}+Q(u;R_0)\big]\frac{r^\alpha}{R_0^\alpha}.\]
\end{theorem}
\begin{proof}
First we consider the case $p\ge 2$. For all integer $j\ge 0$ we set $R_j=R_0/4^j$, $B_j=B_{R_j}$, and $\frac{1}{2} B_j=B_{R_j/2}$. We put forward the following
\vskip2pt
\noindent
{\em Claim.} There exist a universal $\alpha\in(0,1)$ and a real $\lambda>0$ (depending on all the data), a non-decreasing sequence $(m_j)$, and a non-increasing sequence $(M_j)$, such that for all $j\ge 0$
\[m_j\le\inf_{B_j}u\le\sup_{B_j}u\le M_j, \quad  M_j-m_j=\lambda R_j^\alpha.\]
We argue by induction on $j$. Step zero: we set $M_0=\sup_{B_0}u$ and $m_0=M_0-\lambda R_0^\alpha$, where $\lambda>0$ satisfies
\beq\label{io1}
\lambda\ge \frac{2\|u\|_{L^\infty(B_0)}}{R_0^\alpha},
\eeq
which clearly implies
\[\inf_{B_0}u\ge m_0.\]
Inductive step: assume that sequences $(m_j)$, $(M_j)$ are constructed up to the index $j$. Then
\beq\label{iox}
\begin{split}
M_j-m_j &= \dashint_{B_j\setminus\frac{1}{2} B_j}(M_j-u)\,dx+\dashint_{B_j\setminus\frac{1}{2}B_j}(u-m_j)\,dx \\
&\le \Big(\dashint_{B_j\setminus\frac{1}{2}B_j}(M_j-u)^{p-1}\,dx\Big)^\frac{1}{p-1}+\Big(\dashint_{B_j\setminus\frac{1}{2} B_j}(u-m_j)^{p-1}\,dx\Big)^\frac{1}{p-1}.
\end{split}
\eeq
Let $\sigma\in(0,1)$, $\tilde C>0$ be as in Lemma \ref{harnackloc}, and multiply the previous inequality by $\sigma$ to obtain, via \eqref{hloctesi}, 
\beq
\label{ioxx}
\begin{split}
\sigma(M_j-m_j) &\le \inf_{B_{j+1}}(M_j-u)+\inf_{B_{j+1}}(u-m_j)+2\tilde C(KR_0^{ps})^\frac{1}{p-1} \\
&\quad + C_\eps\big[{\rm Tail}((M_j-u)_-;R_j)+{\rm Tail}((u-m_j)_-;R_j)\big]\\
&\quad +\eps\Big[\sup_{B_j}(M_j-u)+\sup_{B_j}(u-m_j)\Big].
\end{split}
\eeq
Setting universally $\eps=\sigma/4$, $C=\max\{2\tilde C,C_\eps\}$ and rearranging, we have
\beq\label{io2}
\begin{split}
\underset{B_{j+1}}{\rm osc}\,u &\le \Big(1-\frac{\sigma}{2}\Big)(M_j-m_j)\\
&\quad +C\big[(KR_0^{ps})^\frac{1}{p-1}+{\rm Tail}((M_j-u)_-;R_j)+{\rm Tail}((u-m_j)_-;R_j)\big].
\end{split}
\eeq
Now we provide an estimate of both non-local tails, firstly noting that
\beq\label{io3}
\begin{split}
{\rm Tail}((u-m_j)_-;R_j)^{p-1}&=R_j^{ps}\sum_{k=0}^{j-1}\int_{B_k\setminus B_{k+1}}\frac{(u(y)-m_j)_-^{p-1}}{|y|^{N+ps}}\,dy\\
&\quad +R_j^{ps}\int_{B^c_0}\frac{(u(y)-m_j)_-^{p-1}}{|y|^{N+ps}}\,dy.
\end{split}
\eeq
We consider the first term: by the inductive hypothesis, for all $0\le k\le j-1$ we have in $B_k\setminus B_{k+1}$
\[(u-m_j)_-\le m_j-m_k \le (m_j-M_j)+(M_k-m_k) = \lambda(R_k^\alpha-R_j^\alpha),\]
hence
\begin{align*}
\sum_{k=0}^{j-1}\int_{B_k\setminus B_{k+1}}\frac{(u(y)-m_j)_-^{p-1}}{|y|^{N+ps}}\,dy &\le \lambda^{p-1}R_j^{\alpha(p-1)}\sum_{k=0}^{j-1}\int_{B_k\setminus B_{k+1}}\frac{(4^{\alpha(j-k)}-1)^{p-1}}{|y|^{N+ps}}\,dy \\
&\le C\lambda^{p-1} S(\alpha)R_j^{\alpha(p-1)-ps},
\end{align*}
where we have set for all $\alpha\in(0,1)$
\[S(\alpha)=\sum_{h=1}^\infty\frac{(4^{\alpha h}-1)^{p-1}}{4^{psh}},\]
noting that $S(\alpha)\to 0$ as $\alpha\to 0^+$. Regarding the second term, by the inductive hypothesis we have
\[m_j\le\inf_{B_j}u\le\sup_{B_j}u\le\|u\|_{L^\infty(B_0)},\]
hence
\[\int_{B^c_0}\frac{(u(y)-m_j)_-^{p-1}}{|y|^{N+ps}}\,dy\le \int_{B^c_0}\frac{(\|u\|_{L^\infty(B_0)}+|u(y)|)^{p-1}}{|y|^{N+ps}}\,dy \le \frac{CQ(u; R_0)^{p-1}}{R_0^{ps}}.\]
Choosing $\alpha<ps/(p-1)$ and plugging the above inequalities in \eqref{io3}, we get
\begin{align*}
{\rm Tail}((u-m_j)_-;R_j) &\le C\Big[\lambda^{p-1}S(\alpha)R_j^{\alpha(p-1)}+\frac{Q(u; R_0)^{p-1}R_j^{ps}}{R_0^{ps}}\Big]^\frac{1}{p-1} \\
&\le C\Big[\lambda S(\alpha)^\frac{1}{p-1}+\frac{Q(u; R_0)}{R_0^\alpha}\Big]R_j^\alpha.
\end{align*}
An analogous estimate holds for ${\rm Tail}((M_j-u)_-;R_j)$, so from \eqref{io2} we have
\beq\label{io4}
\begin{split}
\underset{B_{j+1}}{\rm osc}\,u &\le \Big(1-\frac{\sigma}{2}\Big)\lambda R_j^\alpha+C\Big[(KR_0^{ps})^\frac{1}{p-1}+\lambda S(\alpha)^\frac{1}{p-1}R_j^\alpha+\frac{Q(u; R_0)R_j^\alpha}{R_0^\alpha}\Big] \\
&\le 4^\alpha\Big[\Big(1-\frac{\sigma}{2}\Big)+CS(\alpha)^\frac{1}{p-1}\Big]\lambda R_{j+1}^\alpha\\
&\quad +4^\alpha C\Big[K^\frac{1}{p-1}R_0^{\frac{ps}{p-1}-\alpha}+\frac{Q(u; R_0)}{R_0^\alpha}\Big]R_{j+1}^\alpha.
\end{split}
\eeq
Now we choose $\alpha\in(0,ps/(p-1))$ universally such that
\[4^\alpha\Big[\Big(1-\frac{\sigma}{2}\Big)+CS(\alpha)^\frac{1}{p-1}\Big]\le 1-\frac{\sigma}{4},\]
and we set
\beq\label{io5}
\lambda=\frac{4^{\alpha+1}}{\sigma} C\Big[K^\frac{1}{p-1}R_0^{\frac{ps}{p-1}-\alpha}+\frac{Q(u; R_0)}{R_0^\alpha}\Big],
\eeq
which implies \eqref{io1} as $4^{\alpha+1} C/\sigma>2$.
So, \eqref{io4} forces
\[\underset{B_{j+1}}{\rm osc}\,u\le \lambda R_{j+1}^\alpha.\]
We may pick $m_{j+1}$, $M_{j+1}$ such that
\[m_j\le m_{j+1}\le\inf_{B_{j+1}}u\le\sup_{B_{j+1}}u\le M_{j+1}\le M_j, \ M_{j+1}-m_{j+1}=\lambda R_{j+1}^\alpha,\]
which completes the induction and proves the claim.
\vskip2pt
\noindent
Now fix $r\in(0,R_0)$ and find an integer $j\ge 0$ such that $R_{j+1}\le r<R_j$, thus $R_j\leq 4r$. Hence, by the claim and \eqref{io5}, we have
\[\underset{B_r}{\rm osc}\,u \le \underset{B_j}{\rm osc}\,u \le \lambda R_j^\alpha \le C[(KR_0^{ps})^\frac{1}{p-1}+Q(u; R_0)\big]\frac{r^\alpha}{R_0^\alpha},\]
which concludes the argument.
\vskip2pt
\noindent
Now we consider the case $p\in(1,2)$. The only major difference is in \eqref{io2}: instead of \eqref{iox} we use the inductive hypothesis to see that 
\[M_j-u \le (M_j-m_j)^{2-p}(M_j-u)^{p-1},\quad \text{in $B_j$},\]
and similarly for $u-m_j$. Hence
\[M_j-m_j\le(M_j-m_j)^{2-p}\Big[\dashint_{B_j\setminus\frac{1}{2}B_j}(M_j-u)^{p-1}\,dx+\dashint_{B_j\setminus\frac{1}{2}B_j}(u-m_j)^{p-1}\,dx\Big],\]
which in turn implies through \eqref{in3}
\begin{align*}
M_j-m_j&\le\Big[\dashint_{B_j\setminus\frac{1}{2}B_j}(M_j-u)^{p-1}\,dx+\dashint_{B_j\setminus\frac{1}{2}B_j}(u-m_j)^{p-1}\,dx\Big]^\frac{1}{p-1}\\
&\leq 2^{\frac{2-p}{p-1}}\Big[\Big(\dashint_{B_j\setminus\frac{1}{2} B_j}(M_j-u)^{p-1}\,dx\Big)^{\frac{1}{p-1}}+\Big(\dashint_{B_j\setminus\frac{1}{2} B_j}(u-m_j)^{p-1}\,dx\Big)^{\frac{1}{p-1}}\Big].
\end{align*}
Multiplying by $\sigma/2^{\frac{2-p}{p-1}}$ and applying Lemma \ref{harnackloc} we obtain \eqref{ioxx} with $\tilde\sigma=\sigma/2^{\frac{2-p}{p-1}}$, and the proof follows {\em verbatim}.
\end{proof}

\noindent
The next corollary of Theorem \ref{osc} follows from standard arguments.

\begin{corollary}
\label{corlocalpha}
There exists universal $C>0$ and $\alpha\in (0,1)$ with the following property: for all $u\in \widetilde{W}^{s,p}(B_{2R_0}(x_0))\cap L^\infty(B_{2R_0}(x_0))$ 
such that $|\psl u|\le K$ weakly in $B_{2R_0}(x_0)$,
\beq
\label{tesicorloc}
[u]_{C^\alpha(B_{R_0}(x_0))}\le C\big[(KR_0^{ps})^\frac{1}{p-1}+Q(u; x_0, 2R_0)\big]R_0^{-\alpha}.
\eeq
\end{corollary}

\begin{proof}
Given $x, y$ in $B_{R_0}(x_0)$, let $r=|x-y|$. It suffices to apply Theorem \ref{osc} to the ball $B_{R_0}(x)\subseteq B_{2R_0}(x_0)$. Clearly $\|u\|_{L^\infty(B_{R_0}(x))}\leq \|u\|_{L^\infty(B_{2R_0}(x_0))}$ and 
\begin{align*}
&\T(u; x, R_0)^{p-1}\\
&=R_0^{ps}\int_{B_{R_0}^c(x)}\frac{|u(y)|^{p-1}}{|x-y|^{N+ps}}\, dy\\
&\leq CR_0^{ps}\Big[\int_{B_{2R_0}(x_0)\setminus B_{R_0}(x)}\frac{\|u\|_{L^\infty(B_{2R_0}(x_0))}^{p-1}}{|x-y|^{N+ps}}\, dy+\int_{B_{2R_0}^c(x_0)}\frac{|u(y)|^{p-1}}{|x-y|^{N+ps}}\, dy\Big]\\
&\leq C\|u\|_{L^\infty(B_{2R_0}(x_0))}^{p-1}+CR_0^{ps}\int_{B_{2R_0}^c(x_0)}\frac{|u(y)|^{p-1}}{|x_0-y|^{N+ps}}\, dy
\end{align*}
for a universal $C$, where as usual we used $|x-y|\geq |x_0-y|/2$ for  $y\in B^c_{2R_0}(x_0)$, $x\in B_{R_0}(x_0)$.
This implies that
\[Q(u; x, R_0)\leq CQ(u; x_0, 2R_0),\]
and thus the desired estimate on the H\"older seminorm.
\end{proof}

\subsection{Global H\"older regularity}

We finally prove the stated H\"older regularity result up to the boundary.
\vskip4pt
\noindent
\textbf{Proof of Theorem \ref{main}.} We set $K=\|f\|_{L^\infty(\Omega)}$. Corollary \ref{apb} already provides the desired estimate for the $\sup$-norm, namely
\[\|u\|_{L^\infty(\Omega)}\le CK^\frac{1}{p-1},\]
so we can focus on the H\"older seminorm.
\vskip2pt
\noindent
Let $\alpha$ be the one given in Corollary \ref{corlocalpha}. We can assume $\alpha\in(0,s]$. Through a covering argument, \eqref{tesicorloc} implies that $u\in C^{\alpha}_{\rm loc}(\overline\Omega')$ for all $\Omega'\Subset\Omega$, with a bound of the form
\[\|u\|_{C^\alpha(\overline\Omega')}\leq C_{\Omega'}K^{\frac{1}{p-1}},\quad C_{\Omega'}=C(N, p, s, \Omega, \Omega').\]
Hence it suffices to prove \eqref{thm57tesi} in the closure of a fixed $\rho$-neighbourhood of $\partial\Omega$. We will suppose that $\rho>0$ is so small (depending only on $\Omega$) that Lemma \ref{geo1} holds, and thus the metric projection 
\[\Pi:V\to \partial\Omega, \quad \Pi(x)=\underset{y\in \Omega^c}{\rm Argmin}|x-y|\]
is well defined on $V:=\{x\in\overline\Omega:\delta(x)\leq \rho\}$. We claim that 
\beq
\label{claimfin}
[u]_{C^\alpha(B_{r/2}(x))}\leq C_\Omega K^{\frac{1}{p-1}},\quad \text{for all $x\in V$ and $r=\delta(x)$} 
\eeq
for some constant $C_\Omega=C(N, p, s, \Omega)$, independent on $x\in V$. We recall \eqref{tesicorloc}, which in the present case rephrases (up to a universal constant) as
\[[u]_{C^\alpha(B_{r/2}(x))}\le C\big[(Kr^{ps})^\frac{1}{p-1}+\|u\|_{L^\infty(B_r(x))}+\T(u;x,r)\big]r^{-\alpha}.\]
To prove \eqref{claimfin}, it suffices to bound the three terms on the right hand side of the above inequality. The first one it trivially dealt with since $\alpha\leq s\leq  ps/(p-1)$, and thus
\[r^{-\alpha}(Kr^{ps})^\frac{1}{p-1}\leq K^{\frac{1}{p-1}}\rho^{\frac{ps}{p-1}-\alpha}.\]
For the second one we use Theorem \ref{estid} and $\alpha\leq s$ to get
\[\|u\|_{L^\infty(B_{r}(x))}\leq CK^{\frac{1}{p-1}}(\delta(x)+r)^s\leq CK^{\frac{1}{p-1}}\rho^{s-\alpha}r^\alpha,\]
and thus the claimed bound. Similarly for the last term we employ again \eqref{thm44tesi}, together with 
\[\delta(y)\leq |y-\Pi(x)|\leq |y-x|+|x-\Pi(x)|\leq |y-x|+r\leq 2|x-y|,\quad \forall y\in B_r^c(x),\]
to get
\begin{align*}
\T(u;x,r)^{p-1} &\leq CKr^{ps}\int_{B_{r}^c(x)}\frac{\delta^{s(p-1)}(y)}{|x-y|^{N+ps}}\, dy\\
&\leq CKr^{ps}\int_{B_{r}^c(x)}\frac{|x-y|^{s(p-1)}}{|x-y|^{N+ps}}\, dy\\
&\leq CKr^{s(p-1)}.
\end{align*}
Again due to $\alpha\leq s$ we obtain the claimed bound, and the proof of \eqref{claimfin} is completed.
To prove the theorem, pick $x, y\in V$ and suppose without loss of generality that $|x-\Pi(x)|\geq|y-\Pi(y)|$. Two cases may occur:
\begin{itemize}
\item either $2|x-y|< |x-\Pi(x)|$, in which case we set $r=\delta(x)$ and apply \eqref{claimfin} in $B_{r/2}(x)$, to obtain 
\[|u(x)-u(y)|\leq C K^{\frac{1}{p-1}}|x-y|^\alpha;\]
\item or $2|x-y|\geq |x-\Pi(x)|\geq |y-\Pi(y)|$, in which case \eqref{thm44tesi} ensures
\begin{align*}
|u(x)-u(y)|&\leq |u(x)|+|u(y)|\leq C K^{\frac{1}{p-1}}(\delta^s(x)+\delta^s(y))\\
&=C K^{\frac{1}{p-1}}(|x-\Pi(x)|^s+|y-\Pi(y)|^s)\\
&\leq C K^{\frac{1}{p-1}}|x-y|^s\\
&\leq C K^{\frac{1}{p-1}}\rho^{s-\alpha}|x-y|^\alpha.
\end{align*}
\end{itemize}
Thus in both cases the $\alpha$-H\"older seminorm is bounded in $V$ and the proof is completed. \qed

\begin{remark}
\label{remalpha}
As the proofs above show, interior regularity (Theorem \ref{osc}) forces in particular $\alpha<ps/(p-1)$, while in order to control the behavior of weak solutions near the boundary we need the more restrictive bound $\alpha\le s$. Anyway, our H\"older exponent remains not explicitly determined.
\end{remark}

\projects{\noindent The first and second authors were supported by GNAMPA project ``Problemi al contorno per operatori non locali non lineari''.}

\end{document}